\documentclass[12pt]{amsart}

\usepackage{color}

\usepackage{amsmath,amssymb}
\usepackage{amsthm}
\usepackage{bbold}
\usepackage[hidelinks]{hyperref}
\usepackage[margin=1.0in]{geometry}

\numberwithin{equation}{section}

\newtheorem{prop}{Proposition}
\newtheorem{lemma}[prop]{Lemma}
\newtheorem{mres}[prop]{Main Result}

\newtheorem{thm}[prop]{Theorem}
\newtheorem{cor}[prop]{Corollary}

\numberwithin{prop}{section}

\theoremstyle{definition}
\newtheorem{defn}[prop]{Definition}

\newtheorem{example}[prop]{Example}
\newtheorem{rmk}[prop]{Remark}
\newtheorem{nota}[prop]{Notation}

\newcommand{\lp}[1]{\left( #1 \right)}

\newcommand{\del}{\partial}

\newcommand{\N}{\nabla}

\newcommand{\WW}{\mathcal W}

\renewcommand{\bar}[1]{\overline{#1}}

\DeclareMathOperator{\Aut}{Aut}

\DeclareMathOperator{\R}{R}
\DeclareMathOperator{\Rc}{Rc}
\DeclareMathOperator{\Ric}{Ric}
\DeclareMathOperator{\Rm}{Rm}

\DeclareMathOperator{\tr}{tr}

\DeclareMathOperator{\End}{End}

\begin{document}
\title[Ricci Flows with Nilpotent Symmetry]{Ricci Flows with Nilpotent Symmetry \\ and Zero Bundle Curvature}
\author{Steven Gindi}
\address{Yeshiva University\\
             New York, NY 10033}
\email{\href{mailto:steven.gindi@yu.edu}{steven.gindi@yu.edu}}

\vspace*{-10mm}
\maketitle
\tableofcontents
\begin{large}

\section{Introduction and Summary of Results}

In 2010, Lott studied the long time behavior of three dimensional Ricci flows with bounded   normalized   curvature \cite{LOTT2}.  While a sequence of such flows may collapse, Lott demonstrated   that  an associated Ricci flow limit exists which has local nilpotent symmetry \cite{LOTT1}. (This was based on the work of Cheeger, Fukaya and Gromov, and Hamilton \cite{CFG,Ham1}.) He then analyzed the blowdown limits of these solutions by constructing a functional that is monotone along invariant Ricci flows  on \textit{abelian} principal bundles, which was sufficient for his three dimensional applications.  Consequently, Lott proved, that after passing to the universal cover, a three dimensional Ricci flow, with bounded   normalized  curvature and normalized   diameter,  approaches a homogeneous expanding soliton, a type of self-similar solution \cite{LOTT2}.
 
Generalizing this procedure to higher dimensions requires an analysis of Ricci flow blowdown limits on \textit{nonabelian}, nilpotent principal bundles. The main challenge is to construct a functional that is monotone along the flow on such bundles. 

In this paper, we use Lott's functional and construct a new functional to derive rigidity results for invariant Ricci flow blowdown limits on nilpotent principal bundles with zero associated curvature (see Theorems \ref{THMBLOW} and \ref{THMBLOW2}).  Consequently, we prove that the blowdown limit is locally an expanding Ricci soliton when the structure group is the three dimensional Heisenberg group, $Nil^{3}$ (Theorem \ref{THMSOL1}). In addition, we  classify this soliton  when the base manifold is one dimensional (Theorem \ref{THMFOURSOL}). This, together with Lott's work in the abelian setting \cite{LOTT2}, yields a complete local classification of invariant Ricci flow blowdown limits on four dimensional, nilpotent principal bundles.

To derive these results, we first prove that Lott’s functional, designed specifically for the abelian
setting, contains hidden monotonicity properties
for Ricci flows on nonabelian, nilpotent principal bundles (Definition \ref{DEFWL} and Theorem \ref{THMDWL}).  We then extend Lott's functional by incorporating curvature terms based on the nonabelian structure group (Definition \ref{DEFW2}). We prove that this new functional is monotone for  Ricci flows on $Nil^{3}$ bundles and that it yields stronger blowdown limit results (Theorems \ref{THMDW2} and  \ref{THMBLOW2}).

\subsection{Summary of Results}
We now present a more detailed summary of our results regarding functionals and blowdown limits.
The forthcoming sections will not rely on the definitions and statements made here.

\subsubsection{Background and Notation}
Let $g(t)$ be a Ricci flow on a manifold $M$ for $t\in(0, \infty)$ and let $\{s_{i}\}$ be a sequence of positive numbers such that $\lim_{i \rightarrow \infty}s_{i}= \infty$. Suppose  the sequence of Ricci flows $g_{i}(t):= s_{i}^{-1}g(s_{i}t)$ Cheeger-Gromov converges to $g_{\infty}(t)$ on a manifold $M_{\infty}$ (see Definition \ref{DEFCGCON}). We call  $g_{\infty}(t)$ a sequential blowdown limit of $g(t)$. 

To study Ricci flow blowdown limits  on principal bundles, we will need the following data associated with an invariant metric $\bar{g}$ on a $\mathcal{G}$-principal bundle $P \rightarrow M$ (see Section \ref{SECRFP1}):   
a fiberwise metric $G \oplus g$ on $\mathfrak{G} \oplus TM$, where $\mathfrak{G}$ is the adjoint bundle of $P$, and a connection $D$ on $\mathfrak{G} \oplus TM$. In addition, we have the connection $\bar{A}$ on $P$ that is associated with $\bar{g}$ and the tensor $F:=F_{\bar{A}} \in \Gamma(\wedge^{2}T^{*}M\otimes \mathfrak{G})$ that corresponds to the curvature $\bar{F}:=d\bar{A}+\frac{1}{2}[\bar{A},\bar{A}]$. 
 
We will also need the following definition:

\begin{defn} \label{DEFDEL3}
Let  $\mathfrak{g}$  be a Lie algebra with Lie bracket $[\hspace{.1mm},\hspace{.1mm}]$. Define $\delta: \End \mathfrak{g} \rightarrow \wedge^{2}\mathfrak{g}^{*} \otimes \mathfrak{g}$ by 
\begin{align*}
\delta(A)(\eta_{1},\eta_{2})= A[\eta_{1},\eta_{2}] - [A\eta_{1},\eta_{2}]- [\eta_{1},A\eta_{2}].
\end{align*}
\end{defn}
Note that  $\delta(A)=0$ if and only if $A$ is a derivation of $\mathfrak{g}$. 

\subsubsection{Lott's Functional: New Monotonicity Properties and Rigidity Results}
 Using the above notation, the $\WW_{+}$-functional in Lott's paper \cite{LOTT2} is defined as follows. 

\begin{defn} \label{DEFINTWL} Let $\bar{g}$ be an invariant metric on a $\mathcal{G}$--principal bundle $P\rightarrow M$, which is fibered over a compact, $n$-dimensional manifold. Also let  $f \in C^{\infty}(M)$ and $\tau \in \mathbb{R}_{>0}$.  Define
\begin{equation*}
\mathcal{\WW}_{L,+}(\bar{g},f, \tau)= \int_{M} \Big{(} \tau (|\nabla f|^{2}+R_{g} -\frac{1}{4} |DG|^{2} -\frac{1}{4} |F|^{2}) -f+n \Big{)} \frac{e^{-f}}{(4\pi \tau)^{\frac{n}{2}}}dV_{g}. 
\end{equation*}
\end{defn}

Lott's functional can be derived by applying a Perelman type point of view to an abelian principal bundle \cite{PERLMAN1,FELDMAN1,LOTT2, STREETS4}. At first sight, this functional should not be useful in the nonabelian setting. Yet, using subtle integration by parts, we prove the following result (see Theorem \ref{THMDWL} and Lemma \ref{LEMREL1} for the details): 

\begin{thm} Let $\bar{g}(t)$ be an invariant Ricci flow defined for $t \in (0, \infty)$ on a nilpotent $\mathcal{G}$-principal bundle $P\rightarrow M$, which is fibered over a compact, $n$-dimensional  manifold. Suppose the curvature $F_{\bar{A}(t)}=0$.  Also let $f_{t} \in C^{\infty}(M)$ be a solution to the conjugate heat equation (\ref{EQRC4}). Then $\mathcal{W}_{L,+}(\bar{g}(t),f(t),t)$ is a nondecreasing function of $t$.
\end{thm}

As an application, we prove the following result about blowdown limits (see Theorem \ref{THMBLOW} and Corollary \ref{CORWLCON} for more details and results). For notation, we will be using Definition \ref{DEFDEL3}.
\begin{thm} \label{THMBLOWINTRO1}
Let $\bar{g}(t)$ be an invariant Ricci flow  defined for $t \in (0, \infty)$    on a nilpotent $\mathcal{G}$-principal bundle $P\rightarrow M$, which is fibered over a compact manifold. Suppose  the curvature $F_{\bar{A}(t)}=0$.  The following holds true for a blowdown limit of $\bar{g}(t)$: 
\begin{enumerate}
 \item[1)]  $(D_{\cdot}DG)_{\cdot}- D_{\cdot}G(*,\cdot_{1})D_{\cdot}G(*,\cdot_{1})=0$
 \item[2)]  $Ric_{g}  +\frac{g}{2t} -\frac{1}{4}DG(\cdot_{1},\cdot_{2})DG(\cdot_{1},\cdot_{2}) =0 $
\item[3)] $\delta(G^{-1}DG)=0$.
  \end{enumerate}
Moreover, we have
\begin{enumerate}
\item[4)] $\frac{\del G}{\del t}=-2\Ric_{G}$
\item[5)]$g(t)= tg|_{t=1}$.
\end{enumerate}
\end{thm}

\begin{nota}
As above, we will use $\cdot$'s to  trace over components of a tensor with respect to a metric.
\end{nota}

\begin{rmk}
In \cite{GS3}, we derived Theorems \ref{THMBLOWINTRO1} and \ref{THMBLOW} in the case when $\mathcal{G}$ is $Nil^{3}$ and $dim M=1$ by constructing a functional, $\mathcal{I}$, specifically for this setting.
\end{rmk}

The equations in Parts 1) and  2) are known as the harmonic-Einstein equations \cite{LOTT2}. Moreover, Part 3) is equivalent to the condition that $G^{-1}DG$ is a one form valued derivation when restricted to each fiber of $\mathfrak{G}$.

\subsubsection{Toward a More Complete Functional: Lie Group Case} Theorems \ref{THMBLOWINTRO1} and \ref{THMBLOW}, derived from Lott's functional, are not strong enough to show that the blowdown limit is locally an expanding Ricci soliton, which is expected to be true. The reason is that Part 4) of Theorem \ref{THMBLOWINTRO1} only yields that the limit restricted to each fiber of the principal bundle is some general Ricci flow.  

The aim then is to find a functional that will yield a more complete set of blowdown limit results. As a first step, in Section \ref{SECWLIE2} we construct the following functional for Ricci flows on nilpotent Lie groups. It is composed of both the scalar curvature, $\R_{G}$, and the Ricci curvature, $\Ric_{G}$, of the metric $G$. 

\begin{defn} \label{DEFINTWLIE} Let $\mathcal{G}$ be a nilpotent Lie group with a metric $G$ and let $\tau \in \mathbb{R}$. Define 
\begin{equation*}
\WW_{+}(G,\tau)= \tau \R_{G} +\tau^{2}|\Ric_{G}|^{2}.
\end{equation*}
\end{defn}
We discuss the motivation for the above $\WW_{+}$-functional in Section \ref{SECMOT1}. It is a natural generalization of a functional we built for $Nil^{3}$ and is also inspired by work of Lauret  \cite{LAURET}.

In Proposition \ref{PROPDW}, we prove that this functional is monotone along invariant Ricci flow solutions on nilpotent Lie groups. We then derive the following result about blowdown limits (Proposition \ref{PROPBLOWLIE3}). For some notation, we  use Definition \ref{DEFDEL3} and define $\Rc_{G}=G^{-1}\Ric_{G}$. 
\begin{prop}
An invariant Ricci flow blowdown limit $G_{\infty}(t)$ on a nilpotent Lie group $\mathcal{G}$ satisfies $\delta(\Rc_{G_{\infty}}+ \frac{\mathbb{1}}{2t})=0$.
\end{prop}

An invariant metric on $\mathcal{G}$ that satisfies the condition in the proposition is an example of  a Ricci nilsoliton,  which is  locally  an expanding Ricci soliton (see for example \cite{LAURET} and Section \ref{SECNILSOL3}). We thus have (Proposition \ref{PROPBLOWLIE3}):
\begin{cor}\label{CORINTROLIEBL} Let $G_{\infty}(t)$ be an invariant Ricci flow blowdown limit on a nilpotent Lie group $\mathcal{G}$. The lift of $G_{\infty}(t)$ to the universal cover of $\mathcal{G}$ is an expanding Ricci soliton. 
\end{cor}
Corollary \ref{CORINTROLIEBL} is known in the literature through other methods (see \cite{BOHM} and references therein).

\subsubsection{ The $\mathcal{W}_{+}$-Functional for Ricci Flows on Nilpotent Principal Bundles} 
In Section \ref{SECW3}, we introduce  the following $\mathcal{W}_{+}$-functional  to study Ricci flow blowdown limits on nilpotent principal bundles. It is an extension of Lott's functional (Definition \ref{DEFINTWL}) by terms similar to those in the $\mathcal{W}_{+}$-functional for Lie groups (Definition \ref{DEFINTWLIE}).

\begin{defn} Let $\bar{g}$ be an invariant metric on a $\mathcal{G}$--principal bundle $P\rightarrow M$ which is fibered over a compact, $n$-dimensional manifold. Also let  $f \in C^{\infty}(M)$, $a \in \mathbb{R}$ and $\tau \in \mathbb{R}_{>0}$.  Define
\begin{align*}
 \WW_{+}(\bar{g},f, a, \tau) = \mathcal{W}_{L,+}(\bar{g},f, \tau) -a \int_{M} \Big{(}\tau R_{G}+ \tau^{2}|Ric_{G}|^{2}\Big{)} \frac{e^{-f}}{(4\pi \tau)^{\frac{n}{2}}}dV_{g}. 
\end{align*}
\end{defn}

We have the following monotonicity  result for Ricci flows on $Nil^{3}$-bundles (Theorem \ref{THMDW2}):

\begin{thm} Let $\bar{g}(t)$ be an invariant Ricci flow    defined for $t \in (0, \infty)$   on a  $Nil^{3}$-principal bundle $P\rightarrow M$,  which is fibered over a compact, $n$-dimensional  manifold. Suppose the curvature $F_{\bar{A}(t)}=0$.  Also let $f_{t} \in C^{\infty}(M)$ be a solution to the conjugate heat equation (\ref{EQRC4}) and let $a\in [0,1]$. Then $\mathcal{W}_{+}(\bar{g}(t),f(t),a,t)$ is a nondecreasing function of $t$.
\end{thm}

Using this, we prove the following stronger result for blowdown limits (Theorem \ref{THMBLOW2}). For notation, we use Definition \ref{DEFDEL3}.
\begin{thm}\label{THMINTBL1}
Let $\bar{g}_{\infty}(t)$ be a blowdown limit of an invariant Ricci flow solution $\bar{g}(t)$ on a $Nil^{3}$-principal bundle with $F_{\bar{A}(t)}=0$.  In addition to the results of Theorems \ref{THMBLOWINTRO1} and \ref{THMBLOW}, the following holds true: $\delta(\Rc_{G_{\infty}} + \frac{\mathbb{1}}{2t})=0$. 
\end{thm}
This extra condition is the missing ingredient needed to prove  (Theorem \ref{THMSOL1} and Lemma \ref{LEMSOL1}):

\begin{thm}\label{THMINTBL12}
The Ricci flow blowdown limit in Theorem \ref{THMINTBL1} is locally an expanding Ricci soliton.
\end{thm}

An important case for which the  $F_{\bar{A}(t)}=0$ assumption in Theorem \ref{THMINTBL1} holds true is when the base manifold of the $Nil^{3}$-principal bundle is one dimensional.    Applying Theorem \ref{THMINTBL12} and Theorem \ref{THMFOURSOL} in Section \ref{SECBLOWSOL2} yields 

\begin{thm}\label{THMINTBL13}
A blowdown limit of an invariant Ricci flow on a  $Nil^{3}$-principal bundle  fibered over $S^{1}$     is locally an expanding Ricci soliton. Moreover, the soliton can be classified as in Theorem \ref{THMFOURSOL}.
\end{thm}

We then have 
\begin{mres}
Theorems \ref{THMINTBL13} and \ref{THMFOURSOL}, combined with Lott's results in the abelian setting \cite{LOTT2}, yield a complete local classification of invariant Ricci flow blowdown limits on four dimensional, nilpotent principal bundles fibered over compact base manifolds.
\end{mres}

For future work, we are currently analyzing the monotonicity of  $\mathcal{W}_{+}$ and other functionals along any invariant Ricci flow  on nilpotent principal bundles.

We now begin by providing background on Ricci flow blowdown limits and Ricci solitons. 

\subsubsection{Acknowledgments}
I thank John Lott for helpful discussions. I also thank Jeffrey Streets for introducing me to Lott's papers \cite{LOTT1,LOTT2} and for previous discussions.

\section{Background: Ricci Flow Blowdown Limits and Expanding Solitons}

\subsection{ Ricci Flow Blowdown Limits}
A Ricci flow solution $g(t)$ on a manifold $M$ is a time dependent Riemannian metric that satisfies
\begin{align}
\frac{\partial g }{\partial t}= - 2\Ric_{g(t)},
\end{align}
where $\Ric_{g(t)}$ is the Ricci curvature of $g(t)$.

In defining Ricci flow blowdown limits, we will be using the following type of convergence for a sequence of  time dependent metrics \cite{Ham1}.

\begin{defn}\label{DEFCGCON}
For each $i \in \mathbb{N}$, let $M_{i}$ be a manifold equipped with a time dependent Riemannian metric $g_{i}(t)$, for $t \in I \subset \mathbb{R}$, and let $m_{i} \in M_{i}$. $(M_{i},m_{i}, g_{i}(t))$ Cheeger-Gromov converges to a pointed manifold with a time dependent metric $(M_{\infty},m_{\infty}, g_{\infty}(t))$ if there exists

$\bullet$ an exhaustion $\{U_{i}\}$  of open subsets of $M_{\infty}$

$\bullet$ for each $i$, a diffeomorphism $\phi_{i}: U_{i} \rightarrow  V_{i}$, where $V_{i}:=\phi_{i}(U_{i})$ is open in  \hspace*{5.5mm} $M_{i}$ and  $\phi_{i}(m_{\infty})= m_{i}$  \\
such that 
$\phi_{i}^{*}g_{i}(t)$ converges to $g_{\infty}(t)$
 smoothly and uniformly on compact subsets of $M_{\infty} \times I$.
\end{defn}

We then have

\begin{defn} \label{DEFBLOW}
Let $g(t)$ be a Ricci flow solution on a manifold $M$. Let $\{s_{i}\}$ be a sequence of  positive numbers such that $\lim_{i \rightarrow \infty} s_{i}=\infty$.  Suppose the sequence of Ricci flows $g_{i}(t):= s_{i}^{-1}g(s_{i}t)$ Cheeger-Gromov converges to  $g_{\infty}(t)$ on a manifold $M_{\infty}$. We call $g_{\infty}(t)$ a sequential Ricci flow blowdown limit of $g(t)$. 
\end{defn}

\subsection{Expanding Ricci Solitons} \label{SECSOL1}

We now provide background on expanding Ricci solitons and self-similar Ricci flow solutions. 

First recall, 
\begin{defn} A  Ricci soliton metric on a manifold $M$ is a Riemannian metric $g_{0}$ that satisfies
\begin{align} \label{EQRS1}
Ric_{g_{0}} +\frac{c g_{0}}{2} + \frac{\mathcal{L}_{V}g_{0}}{2}=0,
\end{align}
for some vector field $V$ on $M$ and $c\in \mathbb{R}$. In the case when $c>0$, we call $g_{0}$ an expanding  soliton metric.
\end{defn}

Let $(M,g_{0})$ be an expanding  soliton metric that satisfies (\ref{EQRS1}) with $c>0$. Suppose, for $t>0$, $V(t):=\frac{V}{ct}$ is generated by a family of diffeomorphisms $\psi_{t}$ on $M$, normalized by $\psi_{\frac{1}{c}}=\mathbb{1}$.  Then for  $t>0$,
\begin{align}
\label{EQSSRF1} g(t)=ct \psi_{t}^{*}g_{0}
\end{align}
is an expanding, self-similar Ricci flow solution that satisfies
\begin{align}
Ric_{g(t)} +\frac{g(t)}{2t} + \frac{\mathcal{L}_{V(t)}g(t)}{2}=0,
\end{align}
with $g(\frac{1}{c})=g_{0}$.
\begin{nota}
We will call a self-similar Ricci flow solution of the form (\ref{EQSSRF1}), where $c$ can be taken to equal one, an expanding Ricci soliton. 
\end{nota}
\section{Background: Ricci Curvature and Nilpotent Lie Groups}

\subsection{The $\delta$-map}
Given a real vector space $V$, consider the action of $GL(V)$ on $\wedge^{2}V^{*} \otimes V$:
\begin{align*}
A \cdot \mu= A\mu(A^{-1}*, A^{-1}*).
\end{align*}
Differentiating, we obtain the Lie algebra homomorphism 
$\psi: \mathfrak{gl}(V) \rightarrow \mathfrak{gl}(\wedge^{2}V^{*} \otimes V)$,
 given by 
\begin{align}
\psi(A)\mu= A\mu(*,*)- \mu(A*,*)- \mu(*,A*).
\end{align}
For a fixed $\mu \in \wedge^{2}V^{*} \otimes V$, we define $\delta_{\mu}: \End V \rightarrow \wedge^{2}V^{*} \otimes V$ by
\begin{align} \label{DEFDEL}
\delta_{\mu}(A)= \psi(A) \mu. 
\end{align}

\begin{nota} \label{RMKDEL}
When $V$ is a Lie algebra $\mathfrak{g}$ with Lie bracket $[\hspace{.1mm},\hspace{.1mm}]$, we will denote $\delta_{[\hspace{.1mm}, \hspace{.1mm}]}$ by $\delta$. In this case, the kernel of $\delta$ is the set of derivations of $\mathfrak{g}$. 
\end{nota}

In the following, we list some basic properties of $\psi$ and $\delta_{\mu}$ when $V$ is equipped with an inner product. For this, we will be using the natural induced inner product on $\wedge^{2}V^{*} \otimes V$ and will also use a $``*"$  to denote the adjoint of a linear map.

\begin{lemma} \label{LEMPSI} Let $V$ be a real vector space with an inner product. Given $A,B \in \End V$ and $\mu_{1},\mu_{2} \in \wedge^{2}V^{*} \otimes V$, we have
\begin{itemize}
 \item[1)] $ \psi(A)^{*}= \psi(A^{*})$
 \item[2)] $\big< \delta_{\mu_{1}}(A^{*}), \delta_{\mu_{2}}(B^{*}) \big>
=\big< \delta_{\mu_{1}}(B), \delta_{\mu_{2}}(A) \big> + \big<\mu_{1}, \delta_{\mu_{2}}([A,B^{*}]) 
    \big>.$
\end{itemize}
\end{lemma}

\subsection{Ricci Curvature}\label{SECBACK1}

Let $\mathcal{G}$ be a nilpotent Lie group with an invariant Riemannian metric $G$. 
The Ricci curvature and the scalar curvature of $G$ are given by:  
   \begin{align}
   \label{EQRIC} \Ric_{G}&= -\frac{1}{2}G([*,\cdot],[*,\cdot]) + \frac{1}{4}G([\cdot_{1},\cdot_{2}],*)G([\cdot_{1},\cdot_{2}],*)
  \\ \label{EQRIC2} \R_{G}&=-\frac{1}{4}|[,]|^{2}.
\end{align}
\begin{nota}\label{NOTADOT}
As in (\ref{EQRIC}), we will use $\cdot$'s to  trace over components of a tensor with respect to a metric.
\end{nota}

We further define
\begin{align} \label{EQRCN}
 &\Rc_{G}=G^{-1}\Ric_{G} \in \Gamma(\End T\mathcal{G}),
\end{align}
so that $G(\Rc_{G}*,*)=\Ric_{G}(*,*)$.
\begin{rmk} \label{RMKRIC1}\mbox{}
\begin{itemize}
\item[1)]Given the $\mathcal{G}$-invariance, we will at times implicitly evaluate the above tensors at the identity element of $\mathcal{G}$.
\item[2)] When $\mathfrak{g}$ is a nilpotent Lie algebra with a metric $G$, we define $\Ric_{G} \in \mathfrak{g}^{*} \otimes \mathfrak{g}^{*}$, $R_{G} \in \mathbb{R}$  and $\Rc_{G} \in \End \mathfrak{g}$ by the same formulas as in (\ref{EQRIC})--(\ref{EQRCN}).
\end{itemize}
\end{rmk}

As in (\ref{EQRCN}), we will be using the following notation.
\begin{nota} \label{NOTATILDE}
Let $V$ be a vector space with an inner product $g$ and let $S \in V^{*} \otimes V^{*}$. We define $g^{-1}S \in \End V$ by setting $g( (g^{-1}S) v_{1},v_{2})=S(v_{1},v_{2})$, for all $v_{i}\in V$. 
\end{nota} 

\begin{example}
Let $G$ be an invariant metric on the three dimensional Heisenberg group $\mathcal{G}=Nil^{3}$. Choose $\{\eta_{i}\}_{i=1,2,3}$ to be an orthonormal basis for $(\mathfrak{g},G)$ such that $\eta_{3}$ spans the center of $\mathfrak{g}$, $Z(\mathfrak{g})$. Using $[\eta_{1}, \eta_{2}] \in Z(\mathfrak{g})$ and (\ref{EQRIC}), one readily shows that 
\begin{align}
 \label{EQRCNIL}&\Rc_{G}= \R_{G} \mathbb{1} \oplus -\R_{G}\mathbb{1}, \text{ based on the splitting }
 \mathfrak{g}=Z(\mathfrak{g}) ^{\perp}\oplus Z(\mathfrak{g})
 \\ \label{EQRCNIL2}&|\Ric_{G}|^{2}= 3\R_{G}^{2}.
 \end{align}
\end{example}

With reference to Notation \ref{RMKDEL}, we derive a basic relation between $\Rc_{G}$ and $\delta$ (see, for example, \cite{LAURET}).   
\begin{lemma} \label{LEMRCDEL1} Let $\mathfrak{g}$ be a nilpotent Lie algebra  with a metric $G$ and let $A \in \End \mathfrak{g}$. The following holds true
\begin{align}
\big< \Rc_{G},A \big> = \frac{1}{4} \big< \delta(A), [,] \big>.
\end{align}
\end{lemma}

\begin{proof}
Using Notation \ref{NOTATILDE}, let $A= G^{-1}S$, for $S \in  \mathfrak{g}^{*} \otimes \mathfrak{g}^{*}$, and consider  
\begin{align*}
\big< \delta(A), [,] \big>
&= G \big( A[\cdot_{1},\cdot_{2}] - [A\cdot_{1},\cdot_{2}]- [\cdot_{1},A\cdot_{2}], [\cdot_{1},\cdot_{2}] \big)
\\&= S([\cdot_{1},\cdot_{2}], [\cdot_{1},\cdot_{2}])
-2G([\cdot_{1},\cdot], [\cdot_{2}, \cdot_{}]) S(\cdot_{1},\cdot_{2}).
\end{align*}
This equals $4 \big< \Ric_{G},S \big>= 4 \big< \Rc_{G},A \big>.$
\end{proof}

When analyzing functionals along the Ricci flow in forthcoming sections, we will need the following lemma. We will continue to use Notation \ref{NOTATILDE} and a $``*"$  to denote the adjoint of a linear map.
\begin{lemma} \label{LEMDRCA}  Let $G(t)$ be a time dependent metric on a nilpotent Lie algebra $\mathfrak{g}$, $S \in \mathfrak{g}^{*} \otimes \mathfrak{g}^{*}$ and $A \in \End \mathfrak{g}$. The following holds true: 
\begin{align} 
\big< \frac{d }{dt}\Ric_{G}, S \big> &=  
\label{LABDRIC1} \frac{1}{4}\big{<} \delta\big( G^{-1}\frac{dG}{dt} \big), \delta(G^{-1}S) \big{>}
+ \big< \Ric_{G}(*,*), \frac{d G}{dt}(*,\cdot)S(\cdot, *) \big> 
\\ \label{LABDRC1} \big<\frac{d}{dt}\Rc_{G}, A &\big>= \frac{1}{4} \big< \delta \big(G^{-1}\frac{d G}{dt} \big), \delta(A^{*}) \big>.
\\ \nonumber   \text{Moreover, we have} 
\\ \label{EQDRIC3} \frac{d}{dt}|\Ric_{G}|^{2}&= \frac{1}{2}\big{<} \delta \big(G^{-1} \frac{dG}{dt} \big), \delta(\Rc_{G}) \big{>}
\\  \label{EQDR}\frac{d}{dt}\R_{G}&=- \big< \Ric_{G},\frac{dG}{dt} \big>.
\end{align}
\end{lemma}

\begin{proof}
To prove (\ref{LABDRC1}), we take the time derivative of both sides of 
\begin{align}
   \big< \Rc_{G}, A \big> =\frac{1}{4} \big< \delta(A), [,] \big>,
\end{align}
which holds true by Lemma \ref{LEMRCDEL1}. To simplify the notation, we set $\hat{\frac{dG}{dt}}:=G^{-1}\frac{dG}{dt}$ and $\hat{S}:= G^{-1}S$.

First consider 
\begin{align}
\nonumber \frac{d}{dt} \big< \Rc_{G}, A  \big>
 &= \big< \frac{d }{dt}\Rc_{G} , A \big>  
  +\frac{dG}{dt} \big(\Rc_{G} \cdot, A \cdot \big)
-G( \Rc_{G} \cdot_{1}, A\cdot_{2} ) \frac{dG}{dt}(\cdot_{1}, \cdot_{2}) 
\\  \nonumber & = \big< \frac{d }{dt}\Rc_{G} , A \big> 
 + \big< \Rc_{G}, [\hat{\frac{dG}{dt}},A] \big>
\\  \label{LABRCA}& = \big< \frac{d }{dt}\Rc_{G} , A \big> 
 + \frac{1}{4}\big< \delta \big( [\hat{\frac{dG}{dt}},A] \big), [,] \big>,   
\end{align}
where we used Lemma \ref{LEMRCDEL1}.

Next, we have
\begin{align}
\nonumber \frac{d}{dt} \big{<}\delta(A), [,] \big{>}
 & = \frac{dG}{dt} \big{(} \delta(A)(\cdot_{1},\cdot_{2}), [\cdot_{1},\cdot_{2}]\big{)}
 -2G \big{(} \delta(A)(\cdot_{1}, \cdot_{3}), [\cdot_{2},\cdot_{3}] \big{)} \frac{dG}{dt}(\cdot_{1},\cdot_{2})
 \\  \label{LABDELB} &= \big{<}\delta(A), \delta(\hat{\frac{dG}{dt}}) \big{>}.
\end{align}
Equating (\ref{LABRCA}) and (\ref{LABDELB}) with a factor of $\frac{1}{4}$, gives 
\begin{align}
\label{LABDRC2}\big< \frac{d}{dt}\Rc_{G} , A \big> =
 \frac{1}{4}\big{<}\delta(\hat{\frac{dG}{dt}}), \delta(A)  \big{>}
+ \frac{1}{4}\big< \delta \big( [A,\hat{\frac{dG}{dt}}] \big), [,] \big>.
\end{align}
Applying Lemma \ref{LEMPSI} proves (\ref{LABDRC1}).

To prove (\ref{LABDRIC1}), consider 
\begin{align*}
\big< \frac{d}{dt} \Ric_{G}, S \big>
=\big< \frac{d}{dt} \Rc_{G} + \hat{\frac{dG}{dt}}  \Rc_{G}, \hat{S} \big>.
\end{align*}

By (\ref{LABDRC2}), this equals
\begin{align*}
&\frac{1}{4}\big{<} \delta(\hat{\frac{dG}{dt}}), \delta(\hat{S}) \big{>}
+ \big< \Rc_{G}, [\hat{S}, \hat{\frac{dG}{dt}}] \big>
+ \big< \hat{\frac{dG}{dt}} \Rc_{G}, \hat{S} \big>
\\& =\frac{1}{4}\big{<}\delta(\hat{\frac{dG}{dt}}), \delta(\hat{S}) \big{>}
+ \big< \Rc_{G}, \hat{S} \hat{\frac{dG}{dt}} \big>
\\& =\frac{1}{4}\big{<} \delta(\hat{\frac{dG}{dt}}), \delta(\hat{S}) \big{>}
+ \big< \Ric_{G}(*,*), \frac{dG}{dt}(*,\cdot)S(\cdot, *) \big>. 
\end{align*}

To prove (\ref{EQDRIC3}), we have
\begin{align*}
\frac{d}{dt} |\Ric_{G}|^{2}
&= 2 \big<\frac{d}{dt} \Ric_{G}, \Ric_{G} \big>
-2 \big< \Ric_{G}(*,*), \frac{d G}{dt}(*,\cdot) \Ric_{G}(\cdot, *) \big>
 \\ & = \frac{1}{2}\big{<} \delta(\hat{\frac{dG}{dt}}), \delta(\Rc_{G}) \big{>},
\end{align*}
where we used (\ref{LABDRIC1}).

Lastly, to prove (\ref{EQDR}), consider
\begin{align*}
\frac{d}{dt}\R_{G} 
&=-\frac{1}{4}\frac{d}{dt}|[,]|^{2}
\\ &= -\frac{1}{4}\frac{dG}{dt}\big([\cdot_{1},\cdot_{2}], [\cdot_{1},\cdot_{2}] \big)
 +\frac{1}{2}G([\cdot_{1},\cdot_{3}],[\cdot_{2},\cdot_{3}]) \frac{dG}{dt}(\cdot_{1},\cdot_{2})
\\& =-\big<\Ric_{G}, \frac{dG}{dt} \big>.
\end{align*}
\end{proof}

\section{Ricci Flow on Nilpotent Lie Groups: Solitons and the $\mathcal{W}_{+}$-Functional}

In Section \ref{SECWLIE1}, we introduce a functional, $\WW_{+}$, which is monotone along invariant Ricci flows on nilpotent Lie groups. We apply this functional in Section \ref{SECBLOWGROUPS} to prove that a Ricci flow blowdown limit on a nilpotent Lie group is locally an expanding nilsoliton, a type of expanding Ricci soliton. 

This  blowdown limit result has been derived in the literature using other methods. In \cite{LOTT1}, Lott proved that blowdown limits of three and four dimensional homogeneous    Ricci flows    are expanding solitons. The same statement for general homogeneous  flows  was proved in \cite{BOHM}  using the bracket flow (see also references therein).

We first present some background  on nilsolitons (see, for example, \cite{LAURET}).

\subsection{Nilsolitons} \label{SECNILSOL3}

Recalling Notation \ref{RMKDEL} and Remark \ref{RMKRIC1}, we have 
\begin{defn}
A  nilsoliton metric on a nilpotent Lie group $\mathcal{G}$ is an invariant metric $G$ that satisfies
\begin{align}\label{DEFNILSOL}
 \delta(\Rc_{G} + \frac{c\mathbb{1}}{2})=0,
\end{align}
for $c \in \mathbb{R}$. In the case when $c>0$, we call $G$ an expanding nilsoliton metric on $\mathcal{G}$.
\end{defn}

\begin{rmk}
If $G$ satisfies (\ref{DEFNILSOL}) and $\mathcal{G}$ is nonabelian then, by Lemma \ref{LEMRCDEL1}, $c=-\frac{2|\Ric_{G}|^{2}}{\R_{G}}> 0.$
\end{rmk}
The relation between nilsolitons and Ricci solitons is given in the following lemma.
\begin{lemma} \label{LEMSOL2}Let $\mathcal{G}=(\mathbb{R}^{n}, \cdot)$ be a simply connected nilpotent Lie group with Lie algebra $\mathfrak{g}=(\mathbb{R}^{n}, [,])$ and let $G$ be an expanding nilsoliton metric on  $\mathcal{G}$ such that  (\ref{DEFNILSOL}) holds true for $c>0$. Setting $B:=-\big(\Rc_{G}|_{0} + \frac{c\mathbb{1}}{2} \big)$, we have the following:

1) $G$ is an expanding Ricci soliton metric that satisfies
\begin{align}
\Ric_{G} +\frac{c \hspace{.4mm} G}{2} + \frac{\mathcal{L}_{V}G}{2}=0,
\end{align} 
where $V$ is generated by $\phi_{s}= exp(s B )$, $s\in \mathbb{R}$.

2) For $t>0$, $G(t)=c \hspace{.1mm}t\psi_{t}^{*}G$ is an invariant, expanding Ricci soliton that satisfies
\begin{align}
\Ric_{G(t)} +\frac{G(t)}{2t} + \frac{\mathcal{L}_{V(t)}G(t)}{2}=0,
\end{align} 
where $\psi_{t}:= exp\big( \frac{ln(ct)}{c} B \big)$ generates the vector field $V(t)=\frac{V}{ct}$ with $\psi_{\frac{1}{c}}=\mathbb{1}$.
\end{lemma}
\begin{proof}
Since $B=-\Rc_{G}|_{0}- \frac{c\mathbb{1}}{2}$  is a derivation of $\mathfrak{g}$,  
for $s\in \mathbb{R}$ $exp(sB) \in \Aut \mathfrak{g}=\Aut \mathcal{G}$ and generates the vector field $V|_{x}=Bx$ on $\mathcal{G}= \mathbb{R}^{n}$. One may show that $\mathcal{L}_{V}G$ is invariant and at $0 \in \mathcal{G}$,
\begin{align*}
\mathcal{L}_{V}G(x,y)&= G(Bx,y)+G(x,By)
                                \\&  =  -2\Ric_{G}(x,y)- c \hspace{.3mm}G(x,y).
\end{align*} 
 Hence
\begin{align*}
\Ric_{G} +\frac{c \hspace{.3mm}G}{2} + \frac{\mathcal{L}_{V}G}{2}=0
\end{align*}
on $\mathcal{G}$. 

Part 2) of the lemma follows from Section \ref{SECSOL1} and that $V(t)=\frac{V}{ct}$ is generated by $\psi_{t}=exp \big( \frac{ln(ct)}{c}B \big)$ with $\psi_{\frac{1}{c}}= \mathbb{1}$.
\end{proof}
\subsubsection{$Nil^{3}$ Case}

As an example, we show that any invariant metric on the  three dimensional Heisenberg group $\mathcal{G}=Nil^{3}$ is an expanding nilsoliton.  Defining $Z(\mathfrak{g})$ to be the center of $\mathfrak{g}$, we will first characterize the derivations of the Lie algebra of $Nil^{3}$. 

\begin{lemma} \label{LEMDERNIL}Let $\mathcal{G}$ be $Nil^{3}$ and $A \in \End \mathfrak{g}.$ $\delta(A)=0$ is equivalent to the following two conditions: 
\begin{itemize}
\item [1)] $A: Z(\mathfrak{g}) \rightarrow Z(\mathfrak{g})$
\item [2)] $\tr A=2\tr A|_{Z(\mathfrak{g})}$.
\end{itemize}
\end{lemma}
\begin{proof}
Let $\{\eta_{i}\}_{i=1,2,3}$ be a basis for $\mathfrak{g}$ such that $\eta_{3}=[\eta_{1},\eta_{2}]$ and  spans $Z(\mathfrak{g})$. Setting $A\eta_{i}=A_{ji}\eta_{j}$, we  consider the components of $\delta(A)$.

First, 
\begin{align}
\nonumber \delta(A)(\eta_{3}, \eta_{j})
&= A[\eta_{3}, \eta_{j}]- [A\eta_{3}, \eta_{j}] - [\eta_{3}, A\eta_{j}]
\\  & = -[A\eta_{3}, \eta_{j}]
  = -A_{i3}[\eta_{i},\eta_{j}].
\end{align}

Next,  we have
\begin{align}
\delta(A)(\eta_{1}, \eta_{2})
 \nonumber &= A[\eta_{1}, \eta_{2}]- [A\eta_{1}, \eta_{2}] - [\eta_{1}, A\eta_{2}]
\\& = A_{13}\eta_{1}+ A_{23}\eta_{2}+ \big( A_{33}-A_{11}-A_{22} \big) \eta_{3}.
\end{align}

These components are zero if and only if   
\begin{align*}
A_{13}&=0 \text{ and } A_{23}=0
\\ A_{33}&=A_{11}+ A_{22},
\end{align*}
which in turn are equivalent to 1) and 2) in the lemma.
\end{proof}

We then have
\begin{lemma} \label{LEMNILDEL}
Let $G$ be an invariant metric on $\mathcal{G}=Nil^{3}$.  The following holds true:
\begin{align}
\delta(\Rc_{G} -3\R_{G}\mathbb{1})=0.
 \end{align}
\end{lemma}
\begin{proof}We show that $A:=\Rc_{G} -3\R_{G} \mathbb{1}$ satisfies the conditions in Lemma \ref{LEMDERNIL}. Using (\ref{EQRCNIL}), $A: Z(\mathfrak{g}) \rightarrow Z(\mathfrak{g})$ and
\begin{align*}
\tr A= -8\R_{G}= 2 \tr A|_{Z(\mathfrak{g})}. 
\end{align*}

 \end{proof}

\subsection{$\mathcal{W}_{+}$-Functional for Nilpotent Lie Groups} \label{SECWLIE1}
In this section, we introduce the $\WW_{+}$-functional for Ricci flows on nilpotent Lie groups (Definition \ref{DEFWLIE}). We first motivate it  by analyzing Ricci flows on $Nil^{3}$. 

\subsubsection{Motivation}\label{SECMOT1}
 A first step in finding a scale invariant, monotone functional for $\mathcal{G}=Nil^{3}$ is to consider $tR_{G}$:
\begin{lemma}\label{PROPBRACK1} Given an invariant Ricci flow  $G(t)$ on $Nil^{3}$, $t \R_{G}$ is monotone.  It satisfies
\begin{align}
\frac{d}{dt} t\R_{G}= 6t\R_{G}\Big{(} R_{G}+ \frac{1}{6t} \Big{)} \leq 0. 
\end{align}
\end{lemma}

This follows from 
\begin{align}
\frac{d}{dt} \R_{G}= 2|\Ric_{G}|^{2}= 6\R_{G}^{2},
\end{align}
where we used (\ref{EQDR}) and (\ref{EQRCNIL2}).

Although $t\R_{G}$ is monotone for $Nil^{3}$, it is natural to find a functional whose monotonicity does not depend on the relation $R_{G}+ \frac{1}{6t} \geq 0$. For this, set $q(t)= t\R_{G}+ 3t^{2}R_{G}^{2}$.
\begin{lemma} Let $G(t)$ be an invariant Ricci flow  on $Nil^{3}$. We have
\begin{align*}
\frac{d}{dt}q(t) = 36t^{2}\R_{G}\big{(}\R_{G}+\frac{1}{6t} \big{)}^{2}.
\end{align*}
\end{lemma}

Hence $q(t)$ gives a scale invariant functional that is monotone along  Ricci flows on $Nil^{3}$.

To generalize this result, we observe that using (\ref{EQRCNIL2}) for $Nil^{3}$,
\begin{align*}
 q(t)&=t\R_{G}+ 3t^{2}\R_{G}^{2}
        \\ &=t\R_{G}+ t^{2} |\Ric_{G}|^{2}.
\end{align*}

This, together with  Lauret's work on the bracket flow \cite{LAURET}, was our motivation for introducing the following functional.

\subsubsection{The $\mathcal{W}_{+}$-Functional}\label{SECWLIE2}
\begin{defn}\label{DEFWLIE}Let $\mathcal{G}$ be a nilpotent Lie group with a metric $G$ and let $\tau \in \mathbb{R}$. Define 
\begin{equation*}
\WW_{+}(G,\tau)= \tau \R_{G} +\tau^{2}|\Ric_{G}|^{2}.
\end{equation*}
\end{defn}
This functional is scale invariant: $\WW_{+}(c \hspace{.2mm}G,c\tau)=\WW_{+}(G,\tau),$ for $c>0$.
\begin{rmk}
In the above definition, $G$ is not necessarily invariant and may be defined only on an open subset of $\mathcal{G}$. In the case that $G$ is invariant, we will implicitly consider $\WW_{+}(G,\tau) \in \mathbb{R}$ by evaluating it at the identity element of $\mathcal{G}$.
\end{rmk}
We now analyze the $\mathcal{W}_{+}$-functional along a Ricci flow. We will use the $\delta$-map defined in (\ref{DEFDEL}) and Notation \ref{RMKDEL}.
\begin{prop}\label{PROPDW}
Let $(\mathcal{G},G(t))$ be an invariant Ricci flow on a nilpotent Lie group. The following holds true:
\begin{align}
\frac{d}{dt}\WW_{+}(G(t),t)&= -t^{2}|\delta(\Rc_{G}+ \frac{\mathbb{1}}{2t})|^{2}.
\end{align}
 
\end{prop}

\begin{proof}
By setting $\frac{dG}{dt}=-2\Ric_{G}$ in (\ref{EQDR}) and (\ref{EQDRIC3}), we obtain  
\begin{align}
\frac{d}{dt}\R_{G} &=2|\Ric_{G}|^{2}
\\ \label{EQDRIC4} \frac{d}{dt}|\Ric_{G}|^{2}
& =-|\delta(\Rc_{G})|^{2}.
\end{align}
We then have
\begin{align}
\nonumber \frac{d}{dt}(t\R_{G}+t^{2}|\Ric_{G}|^{2})
&= \R_{G} + t\frac{d}{dt}\R_{G}
   +2t|\Ric_{G}|^{2} +t^{2}\frac{d}{dt} |\Ric_{G}|^{2}
\\\nonumber &=\R_{G} +4t|\Ric_{G}|^{2} -t^{2}|\delta(\Rc_{G})|^{2}
\\ &=-t^{2} \Big( \frac{|[,]|^{2}}{4t^{2}}-\frac{1}{t} \big< \delta(\Rc_{G}), [,] \big> + |\delta(\Rc_{G})|^{2} \Big) 
 \\ \nonumber &=-t^{2}|\delta(\Rc_{G}+ \frac{\mathbb{1}}{2t})|^{2},
\end{align}
where we used Lemma \ref{LEMRCDEL1} to obtain the second term in the third equality. 
\end{proof}
\begin{rmk}
In \cite{LAURET}, Lauret derived (\ref{EQDRIC4}) using the bracket flow.
\end{rmk} 

\subsection{Blowdown Limit Results} \label{SECBLOWGROUPS}
Using Proposition \ref{PROPDW}, we obtain the following results about blowdown limits on nilpotent Lie groups.

\begin{prop}\label{PROPBLOWLIE3}
Let $(\mathcal{G},G(t))$ be an invariant Ricci flow solution on a nilpotent Lie group, for $t\in (0,\infty)$. Let $\{s_{i}\}$ be a sequence of positive numbers such that $\lim s_{i}= \infty$ and define $G_{i}(t)=s_{i}^{-1}G(s_{i}t)$. Suppose that $(\mathcal{G},a_{i},G_{i}(t))$ pointed Cheeger-Gromov converges to an invariant Ricci flow on a nilpotent Lie group, $(\mathcal{G}_{\infty},a_{\infty},G_{\infty}(t))$, where $a_{i}\in \mathcal{G}$ and $a_{\infty}\in \mathcal{G}_{\infty}$.  The following holds true:

1) $\delta \big{(}\Rc_{G_\infty} +\frac{\mathbb{1}}{2t} \big{)}=0$.

\vspace{3mm}

2) $(\tilde{\mathcal{G}}_{\infty}, \tilde{G}_{\infty}(t))$, the lift of $G_{\infty}(t)$ to the universal cover of $\mathcal{G}_{\infty}$, is an expanding Ricci soliton that satisfies:
\begin{align*}
\Ric_{\tilde{G}_{\infty}} +\frac{\mathcal{L}_{V(t)}\tilde{G}_{\infty}}{2} +\frac{\tilde{G}_{\infty}}{2t}=0,
\end{align*}
for some vector field $V(t)$ on $\tilde{\mathcal{G}}_{\infty}$.

\vspace{3mm}

3) $|\Ric_{G_{\infty}}|^{2}=-\frac{\R_{G_{\infty}}}{2t}. $

\vspace{3mm}

In the case when $\mathcal{G}_{\infty}=Nil^{3}$, $\R_{G_{\infty}}=-\frac{1}{6t}$.

\end{prop}

\begin{proof} Using Definition \ref{DEFCGCON}, choose

$\bullet$ an exhaustion $\{U_{i}\}$  of open subsets of $\mathcal{G}_{\infty}$

$\bullet$ for each $i$, a diffeomorphism $\phi_{i}: U_{i} \rightarrow  V_{i}$, where $V_{i}:=\phi_{i}(U_{i})$ is open in  \hspace*{6mm} $\mathcal{G}$ and  $\phi_{i}(a_{\infty})= a_{i}$  \\
such that 
$\phi_{i}^{*}G_{i}(t)$ converges to $G_{\infty}(t)$
 smoothly and uniformly on compact subsets of $\mathcal{G}_{\infty} \times (0,\infty)$.

To prove Part 1), we first show that $\WW_{+}(G_{\infty}(t),t)$ is constant in $t$. For $t>0$, we have 
\begin{align*}
\mathcal{W}_{+}(G_{\infty}(t),t)|_{a_{\infty}}
&=\lim_{i \rightarrow \infty} \mathcal{W}_{+}(\phi_{i}^{*}G_{i}(t),t)|_{a_{\infty}}
\\& = \lim_{i \rightarrow \infty}  \phi_{i}^{*}\big(\mathcal{W}_{+}(G_{i}(t), t) \big)|_{a_{\infty}}
\\ & =\lim_{i \rightarrow \infty} \mathcal{W}_{+}(G_{i}(t), t)|_{a_{i}}.
\end{align*}
Using the invariance of $G$ and letting  $e \in \mathcal{G}$ be the identity element, this equals
\begin{align*}
  \lim_{i \rightarrow \infty} \mathcal{W}_{+}(G_{i}(t), t)|_{e}
 &= \lim_{i \rightarrow \infty} \mathcal{W}_{+}(s_{i}^{-1}G(s_{i}t), t)|_{e}
\\& =\lim_{i \rightarrow \infty} \mathcal{W}_{+}(G(s_{i}t), s_{i}t)|_{e}
\\& = \lim_{u \rightarrow \infty} \mathcal{W}_{+}(G(u), u)|_{e}.
\end{align*}
	The existence of the last limit follows from the existence of the second to last limit and the fact that $\mathcal{W}_{+}(G(t), t)$ is a nonincreasing function of $t$, which holds true by Proposition \ref{PROPDW}.
Hence $\mathcal{W}_{+}(G_{\infty}(t),t)$ is constant in $t$. Part 1) then follows from Proposition \ref{PROPDW}. Using this result, Part 2) follows from Lemma \ref{LEMSOL2}.

To prove Part 3), using Part 1) and Lemma \ref{LEMRCDEL1}, we have 
\begin{align}
\nonumber |\Ric_{G_{\infty}}|^{2}
&= \frac{1}{4} \big< \delta(\Rc_{G_{\infty}}), [,] \big>
\\ \label{EQRCNIL1}& = \frac{1}{8t}|[,]|_{G_{\infty}}^{2}
=-\frac{\R_{G_{\infty}}}{2t}.
\end{align}

When $\mathcal{G}_{\infty}$ is $Nil^{3}$, by (\ref{EQRCNIL2}), $|\Ric_{G_{\infty}}|^{2}= 3\R^{2}_{G_{\infty}}$.  Combining this with (\ref{EQRCNIL1}), gives $\R_{G_{\infty}}=-\frac{1}{6t}$.
\end{proof}

When $\mathcal{G}_{\infty}$ is $Nil^{3}$, Parts 1) and 3) of the above proposition are equivalent.
\begin{lemma} \label{PROPDELNIL} Let $G(t)$ be a time dependent, invariant metric on $Nil^{3}$. The condition 
$\delta(\Rc_{G}+\frac{\mathbb{1}}{2t})=0$ holds true if and only if $\R_{G}=-\frac{1}{6t}$.
\end{lemma}

\begin{proof} 
By Lemma \ref{LEMNILDEL}, $\delta(\Rc_{G}+\frac{\mathbb{1}}{2t})=0$  if and only if
\begin{align*}
-3\R_{G}[\hspace{.2mm}, \hspace{-.2mm}]= \frac{1}{2t}[\hspace{.2mm}, \hspace{-.2mm}],
\end{align*}
which is equivalent to $\R_{G}=-\frac{1}{6t}$.
\end{proof}

\begin{rmk}
One can verify that $\R_{G_{\infty}}=-\frac{1}{6t}$  in the blowdown limit example for $Nil^{3}$ given in \cite{LOTT1}. 
\end{rmk}

\section{Ricci Flow on Principal Bundles: Background} \label{SECRICCBAC}
In the upcoming sections, we will use functionals to study invariant Ricci flow blowdown limits  on nilpotent principal bundles. 
These functionals will be defined over the base manifold in terms of tensors that correspond to an invariant metric on the principal bundle. In this section, we describe this correspondence and derive a system of PDEs for the tensors that are equivalent to the Ricci flow equations on the bundle.  

\subsection{Invariant Metrics and Associated Data}  \label{SECRFP1}

Let $\pi:P \rightarrow M$ be a $\mathcal{G}$--principal bundle with an invariant Riemannian metric $\bar{g}$ and let $\mathcal{E}=  \mathfrak{G} \oplus TM$, where $\mathfrak{G}$ is the adjoint bundle. We will define certain differential and geometric structures on $\mathcal{E}$
by using a correspondence between sections of $\mathcal{E}$ and invariant vector fields on $P$. To describe it, consider the isomorphism $\lambda: TP \rightarrow \pi^{*}\mathcal{E}$ given by
\begin{align} 
\lambda(Z)=  \{p,\bar{A}Z \} + \pi_{*}Z ,
\end{align}
 where $Z \in T_{p}P$ and $\bar{A}$ is the connection on $P$ associated with $\bar{g}$. If $e \in \Gamma(\mathcal{E})$ then  $\tilde{e}:=\lambda^{-1}\pi^{*}e$ is an invariant section of $TP$, where $\pi^{*}e \in \Gamma(\pi^{*}\mathcal{E})$ is the pullback section of $e$. Moreover, if $Z$ is an invariant section of $TP$ then it is straightforward to show that $Z= \tilde{e}$ for some unique $e \in \Gamma(\mathcal{E})$.  This correspondence extends naturally to invariant sections of $(\otimes^{k}T^{*}P)\otimes (\otimes^{l}TP) $
and sections of $(\otimes^{k}\mathcal{E^{*}})\otimes (\otimes^{l}\mathcal{E})$. 


Using this correspondence, we obtain from $\bar{g}$ a fiberwise metric  $g_{_{\mathcal{E}}}:=G \oplus g$ on $\mathcal{E}= \mathfrak{G}\oplus TM$. We also obtain the following data which we will define below and in Section \ref{SECLAC} (see \cite{GS1} for more details): 

 
$\bullet$ a Lie algebroid bracket $[,]$ on $\Gamma(\mathcal{E})$

$\bullet$ a connection $D$ on $\mathcal{E}$

$\bullet$  an exterior derivative $d_{\mathcal{E}}: \Gamma(\wedge^{k}\mathcal{E^{*}})\rightarrow \Gamma(\wedge^{k+1}\mathcal{E^{*}})$

$\bullet$ a Lie algebroid connection $\nabla: \Gamma(\mathcal{E}) \rightarrow \Gamma(\mathcal{E}^{*}\otimes \mathcal{E})$.

\noindent In addition, we have the tensor $F:=F_{\bar{A}} \in \Gamma(\wedge^{2}T^{*}M\otimes \mathfrak{G})$ that corresponds to the curvature $\bar{F}:=d\bar{A}+\frac{1}{2}[\bar{A},\bar{A}]$.

The bracket on $\Gamma(\mathcal{E})$ is defined by 
\begin{defn}
For $e_{1}, e_{2} \in \Gamma(\mathcal{E})$, define $[e_{1},e_{2}]$ to be the unique section of $\mathcal{E}$ that satisfies \[ \widetilde{[e_{1},e_{2}]}= [\tilde{e_{1}},\tilde{e_{2}}]_{Lie},\] where $[,]_{Lie}$ is the Lie bracket on $\Gamma(TP)$. 
\end{defn}

If we define $ \tau: \mathcal{E}\rightarrow TM$ by $\tau(\eta + v)= v $, where $\eta \in \mathfrak{G}$,  we then have: 

\begin{lemma} \label{p:Liealgprop} Let $e_{1},e_{2} \in \Gamma(\mathcal{E})$ and $f\in C^{\infty}(M)$. The following holds true
\begin{itemize}
\item [a)] $[e_{1},fe_{2}]= f[e_{1},e_{2}] +\tau(e_{1})[f](e_{2})$
\item [b)] $\tau([e_{1},e_{2}]) = [\tau(e_{1}),\tau(e_{2})]_{Lie}$
\item [c)] $[,]$ on $\Gamma(\mathcal{E})$ satisfies the Jacobi identity. 
\end{itemize}
In other words, $\tau:(\mathcal{E}, [,]) \rightarrow (TM, [,]_{Lie})$ is a Lie algebroid structure on $\mathcal{E}$.
\end{lemma}

By Lemma \ref{p:Liealgprop}, we obtain the  connection $D$ on $\mathfrak{G}$ defined by
\begin{align} \label{f:Liealgconn}
D_{v}\eta:= [v,\eta].
\end{align}

\begin{nota} \label{NOTAD1}
Given the above Riemannian metric $g$ on $M$, we will at times extend $D$ to a connection on $\mathcal{E}=\mathfrak{G} \oplus TM$ by setting $D=D \oplus \nabla^{g}$, where $\nabla^{g}$ is the Levi-Civita connection of $g$. 
\end{nota} 

In the following lemma, we list some properties that the Lie algebroid bracket satisfies. For some notation, given $p \in P$ and $x \in \mathfrak{g}$, we define $\{p,x\}$ to be the equivalence class of $(p,x)$ in 
$\mathfrak{G}=P \times \mathfrak{g} / \hspace{-1mm}\sim $, where $(p,x) \sim (p \cdot a, Ad_{a^{-1}}x)$,  for  $a \in \mathcal{G}$.

\begin{lemma} \label{PROPBRA} Let $s:U \rightarrow P$ be a local section of $P$, $x,y \in \mathfrak{g}$ and $v,w \in \Gamma(TM)$. The following holds true
\begin{itemize}
\item[a)] $[,]$ restricts to a Lie bracket on each fiber of $\mathfrak{G}$
\item[b)] $[\{s,x\}, \{s,y\}]=-\{s,[x,y]\}$
\item[c)] $D_{v}\{s,x\} =\{s, [(s^{*}\bar{A})v, x]\}$
\item[d)] $[v,w]= [v,w]_{Lie} - F(v,w)$.
\end{itemize}
\end{lemma}

Moreover, associated with the Lie algebroid $\tau:(\mathcal{E}, [,]) \rightarrow (TM, [,]_{Lie})$ is the exterior derivative $d_{\mathcal{E}}$ which squares to zero:

\begin{defn} \label{DEFD1}
Define $d_{\mathcal{E}}: \Gamma(\wedge^{k}\mathcal{E^{*}})\rightarrow \Gamma(\wedge^{k+1}\mathcal{E^{*}})$ by 
\begin{align*}
d_{\mathcal{E}}\sigma(e_{1},...,e_{k+1})=& \sum_{1\leq i\leq k+1}(-1)^{i-1}\tau(e_{i})[\sigma(e_{1},...,\hat{e_{i}},...,e_{k+1})]   \\  
 & + \sum_{1\leq i<j \leq k+1} (-1)^{i+j}\sigma([e_{i},e_{j}],e_{1},...,\hat{e_{i}},...,\hat{e_{j}},...,e_{k+1}).
\end{align*}
\end{defn}

\begin{lemma}
For $\sigma \in \Gamma(\wedge^{k}\mathcal{E^{*}})$,
\[\widetilde{d_{\mathcal{E}}\sigma}= d\tilde{\sigma}. \] 
\end{lemma}

Given the above data, we present two lemmas which we will use in the upcoming sections. 

For the first, let $Z(\mathfrak{G})=\cup_{m \in M} Z(\mathfrak{G}|_{m})$,  where $Z(\mathfrak{G}|_{m})$ is the center of the Lie algebra $(\mathfrak{G}|_{m}, [,])$, and let $Z(\mathfrak{g})$ be the center of the Lie algebra of $\mathcal{G}$. We then have (see \cite{GS3} for the proof).

\begin{lemma} \label{LEMDZ}
\mbox{}
\begin{itemize}
\item[1)] $D$ restricts to a connection on $Z(\mathfrak{G})$, which is a vector bundle over $M$.
\item[2)] If $x \in Z(\mathfrak{g})$ and $s$ is a local section of $P$ then 
     $\{s,x\}$ is a local section of $Z(\mathfrak{G})$ and $D\{s,x\}=0$.
\end{itemize}
\end{lemma} 

As for the second, we have \cite{GS1}
\begin{lemma} \label{propGRAD} Let $\bar{g}$ be an invariant metric on a nilpotent $\mathcal{G}$--principal bundle $\pi: P \rightarrow M$ and let $\{ x_{i}\}$ be a basis for $\mathfrak{g}$. 
\begin{enumerate}
\item [1)] Given $m \in M$, $\det G(\{p,x_{i} \}, \{p,x_{j} \})= \det G(\{p',x_{i} \}, \{p',x_{j} \})$, for all $p, p' \in \pi^{-1}(m)$.
\item[2)] $h|_{m}=\det G(\{p,x_{i} \}, \{p,x_{j} \})$, for $p \in \pi^{-1}(m)$, defines a global function $h$ on M. 
\item[3)] $DG(\cdot,\cdot)=d(\ln h)$.
\end{enumerate}
\end{lemma}

\subsection{Lie Algebroid Connection and Curvatures of $\bar{g}$} \label{SECLAC}
Given the above data, we obtain the following Lie algebroid connection $\nabla: \Gamma(\mathcal{E}) \rightarrow \Gamma(\mathcal{E}^{*}\otimes \mathcal{E})$: 
\begin{defn} \label{DEFNAB}
Letting $e_{1}, e_{2}\in \Gamma(\mathcal{E})$, define $\nabla_{e_{1}} e_{2}$ to be the unique section of $\mathcal{E}$ that satisfies 
\[\widetilde{\nabla_{e_{1}} e_{2}}= \nabla^{\bar{g}}_{\tilde{e_{1}}}\tilde{e_{2}}, \]  where $\nabla^{\bar{g}}$ is the Levi Civita connection on $TP$ associated with $\bar{g}$.
\end{defn}
For some properties, we have
\begin{lemma} \label{PROPNAB1} Let $e_{i} \in \Gamma(\mathcal{E})$ and $ f \in C^{\infty}(M)$.  Then
\begin{enumerate}
\item[1)] $\nabla_{e_{1}}e_{2} -  \nabla_{e_{2}}e_{1} = [e_{1},e_{2}]$
\item[2)] $\nabla_{e_{1}}(fe_{2})=f \nabla_{e_{1}}e_{2} + \tau(e_{1})[f]e_{2}$
\item[3)] $\tau(e_{1})[g_{_{\mathcal{E}}}(e_{2},e_{3})]= g_{_{\mathcal{E}}}(\nabla_{e_{1}}e_{2},e_{3}) + g_{_{\mathcal{E}}}(e_{2},\nabla_{e_{1}}e_{3})$.
\end{enumerate}
\end{lemma}

Letting $\nabla^{g}$ be the Levi Civita connection of $g$, we have the following decomposition of $\nabla$ \cite{GS1}:
\begin{lemma} \label{LEMNABDEC} Let $\eta$, $\eta_{i} \in \Gamma(\mathfrak{G})$ and $v$, $v_{a} \in \Gamma(TM)$.  The following holds true: 
\begin{enumerate}
\item[1)] $\nabla_{v_{1}}v_{2}=\nabla^{g}_{v_{1}}v_{2} -\tfrac{1}{2}F(v_{1},v_{2})$
\item[2)] $\nabla_{v}\eta= D_{v}\eta +\tfrac{1}{2}G^{-1}D_{v}G(\eta, *) +\tfrac{1}{2}g^{-1}G(F(v,*), \eta)$
\item[3)] $\nabla_{\eta}v= \tfrac{1}{2}G^{-1}D_{v}G(\eta,*)+\tfrac{1}{2}g^{-1}G(F(v,*),\eta)$
\item[4)] $\nabla_{\eta_{1}}\eta_{2}= -\tfrac{1}{2}g^{-1}D_{*}G(\eta_{1},\eta_{2})+ \tfrac{1}{2}[\eta_{1},\eta_{2}] +\tfrac{1}{2}G^{-1}G([*,\eta_{1}],\eta_{2})+\tfrac{1}{2}G^{-1}G([*,\eta_{2}],\eta_{1})$.
\end{enumerate}
\end{lemma}

Given $\nabla$, we define the curvature $\Rm^{\nabla} \in \Gamma(\wedge^{2}\mathcal{E}^{*} \otimes \End \mathcal{E})$ via
\begin{align} \label{f:curvdef}
Rm^{\nabla}(e_{1},e_{2})e_{3}=\nabla_{e_{1}}\nabla_{e_{2}}e_{3} - \nabla_{e_{2}}\nabla_{e_{1}}e_{3}- \nabla_{[e_{1},e_{2}]}e_{3},
\end{align}
for $e_{i} \in \Gamma(\mathcal{E})$.  It corresponds to the usual curvature tensor $\bar{\Rm} \in \Gamma(\wedge^{2}T^{*}P \otimes \End TP)$ of $\bar{g}$ on $P$.

We then define $\Rm^{\nabla}(e_{1},e_{2},e_{3},e_{4})= g_{_{\mathcal{E}}}(\Rm^{\nabla}(e_{1},e_{2})e_{3},e_{4})$. The Ricci and scalar curvatures of $\nabla$, $\Ric^{\nabla}$ and $\R^{\nabla}$, are defined using the appropriate traces with respect to $g_{_{\mathcal{E}}}$.

We will now decompose $\Ric^{\nabla}$ based on the splitting $\mathcal{E}= \mathfrak{G} \oplus TM$. First, we define  the following Ricci and scalar curvatures associated with $G$ which correspond to  the formulas in (\ref{EQRIC}) and (\ref{EQRIC2}). 

\begin{defn} \label{DEFRIC2} Let $G$ be a fiberwise Riemannian metric on the adjoint bundle $\mathfrak{G}$ of a nilpotent $\mathcal{G}$-principal bundle, $P \rightarrow M$. Define $\Ric_{G} \in \Gamma(\mathfrak{G}^{*} \otimes \mathfrak{G}^{*} )$, $\Rc_{G}\in \Gamma(\End \mathfrak{G})$ and $R_{G}\in C^{\infty}(M)$ as follows: 
\begin{align}
   \nonumber \Ric_{G}&=  -\frac{1}{2}G([*,\cdot],[*,\cdot]) +\frac{1}{4}G([\cdot_{1},\cdot_{2}],*)G([\cdot_{1},\cdot_{2}],*)
   \\ \Rc_{G}&=G^{-1} \Ric_{G} 
  \\ \nonumber \R_{G}&=-\frac{1}{4}|[,]|^{2}.
\end{align}
\end{defn}
We then have \cite{GS1}
\begin{lemma} \label{p:Ricci} Let $\eta, \eta_{i} \in \mathfrak{G}$ and $v,v_{i} \in TM$. Assuming $\mathcal{G}$ is nilpotent, the following holds true
\begin{align*}
\Ric^{\nabla}(\eta_{1},\eta_{2})=&\ \Ric_{G}(\eta_{1},\eta_{2}) -\tfrac{1}{2}(D_{\cdot}DG)_{\cdot}(\eta_{1},\eta_{2}) 
-\tfrac{1}{4}D_{\cdot}G(\cdot_{1},\cdot_{1})D_{\cdot}G(\eta_{1},\eta_{2})
\\& +\tfrac{1}{2}D_{\cdot}G(\eta_{1},\cdot_{1})D_{\cdot}G(\cdot_{1},\eta_{2}) 
 +\tfrac{1}{4}G(F(\cdot_{1},\cdot_{2}), \eta_{1})G(F(\cdot_{1},\cdot_{2}), \eta_{2})
\\ \Ric^{\nabla}(\eta,v)=&\ \tfrac{1}{2}G(\eta,D_{\cdot}F(v,\cdot)) 
+\tfrac{1}{2}D_{\cdot}G(\eta,F(v,\cdot))
+ \tfrac{1}{4}G(\eta, F(v,\cdot_{2}))D_{\cdot_{2}}G(\cdot_{1},\cdot_{1}) 
\\& -\tfrac{1}{2}D_{v}G([\cdot,\eta],\cdot)
\\ \Ric^{\nabla}(v_{1},v_{2}) =&\ \Ric_{g}(v_{1},v_{2})
 -\tfrac{1}{2}(D_{v_{1}}DG)_{v_{2}}(\cdot,\cdot) 
+\tfrac{1}{4}D_{v_{1}}G(\cdot_{1},\cdot_{2})D_{v_{2}}G(\cdot_{1},\cdot_{2}) 
\\& -\tfrac{1}{2}G(F(v_{1},\cdot),F(v_{2},\cdot)).
\end{align*}
\end{lemma}
We also have 
\begin{lemma} \label{p:scalar} Assuming $\mathcal{G}$ is nilpotent, the following holds true
\begin{align*}
\R^{\N} =&\ \R_{g}+ \R_G -(D_{\cdot}DG)_{\cdot}(\cdot_{1},\cdot_{1}) 
-\tfrac{1}{4}D_{\cdot}G(\cdot_{1},\cdot_{1})D_{\cdot}G(\cdot_{2},\cdot_{2}) 
\\&\ +\frac{3}{4}|DG|^{2}
-\tfrac{1}{4}|F|^{2}.
\end{align*}
\end{lemma}

\begin{nota} \label{NOTABARR}
At times, we will respectively  identify $|\Ric^{\nabla}|^{2}$ and  $\R^{\nabla}$ with $|\bar{\Ric}|^{2}$ and  $\bar{\R}$, where  $\bar{\Ric}$ and $\bar{\R}$ are the Ricci and scalar curvatures of $\bar{g}$. We will also respectively denote $\Ric^{\nabla}(\pi_{\mathfrak{G}}, \pi_{\mathfrak{G}})$, $\Ric^{\nabla}(\pi_{\mathfrak{G}}, \pi_{TM})$ and $\Ric^{\nabla}(\pi_{TM}, \pi_{TM})$ by $\bar{\Ric}_{ij}$, $\bar{\Ric}_{ia}$ and $\bar{\Ric}_{ab}$. (Here, $\pi_{\mathfrak{G}}$ and  $\pi_{TM}$ are the natural projection maps of $\mathcal{E}= \mathfrak{G} \oplus TM$ onto $\mathfrak{G}$ and $TM$.)
\end{nota}

We will need the following lemma, whose proof is similar to that of Lemma \ref{LEMDRCA}. 
\begin{lemma}Let $S \in \mathfrak{G}^{*} \otimes \mathfrak{G}^{*}$ and $A \in \End \mathfrak{G}$. The following holds true: 
\begin{align}
\label{EQDRICS} \big< D\Ric_{G}, S \big> &=  
 \frac{1}{4}\big{<} \delta(G^{-1}DG), \delta(G^{-1}S) \big{>}
+ \big< \Ric_{G}(*,*), DG(*,\cdot)S(\cdot, *) \big> 
\\\label{EQDRC2} \big<D\Rc_{G}, A \big>&= \frac{1}{4} \big< \delta(G^{-1} DG), \delta(A^{*}) \big>.
\\ \nonumber   \text{Moreover, we have} 
\\ d |\Ric_{G}|^{2}&= \frac{1}{2}\big{<} \delta(G^{-1}DG), \delta(\Rc_{G}) \big{>}
\\ \label{EQDR4}  d\R_{G}&=- \big< \Ric_{G},DG \big>.
\end{align}
\end{lemma}

\subsection{Dimensional Reduction of Invariant Ricci Flow} 
 In this section, we recast an invariant Ricci flow solution on a principal bundle as a solution of a system of PDEs on the base manifold. We also use flows of vector fields on the principal bundle to obtain gauge modified PDEs which will be important in analyzing the functionals defined in Definitions \ref{DEFWL} and \ref{DEFW2}.  

\subsubsection{Invariant Ricci Flow Equations}
Let $\bar{g}(t)$ be a time dependent invariant, Riemannian metric on  a nilpotent $\mathcal{G}$-principal bundle $\pi: P \rightarrow M$ which is fibered over a compact, $n$-dimensional manifold. 
Using the results of the previous subsections for each time $t$, we obtain a time dependent fiberwise metric $g_{_{\mathcal{E}}}(t)= G(t) \oplus g(t)$ on $\mathcal{E}= \mathfrak{G}\oplus TM$ and a time dependent connection $D$ on $\mathfrak{G}$. Following Notation \ref{NOTAD1}, we will extend $D$ to a connection on $\mathcal{E}= \mathfrak{G} \oplus TM$ by setting $D=D \oplus  \nabla^{g}$. 

For some other notation,  $\frac{\partial A}{\partial t}$ will denote the time dependent section of $T^{*}M \otimes \mathfrak{G}$ that is associated with $\frac{\partial \bar{A}}{\partial t}$, where $\bar{A}(t) $ are the connections on $P$ induced by $\bar{g}(t)$.  Similarly, $F(t):= F_{\bar{A}(t)} \in \Gamma(\wedge^{2}T^{*}M\otimes \mathfrak{G})$ corresponds to the curvature $\bar{F}(t):=d\bar{A}+\frac{1}{2}[\bar{A},\bar{A}]$.

As proved in \cite{GS1}, using Lemma \ref{p:Ricci}, $\bar{g}(t)$ is a Ricci flow solution if and only if the following equations hold true:
\begin{align} \label{EQRICCI}
&\ \frac{\partial G}{\partial t}(\eta_{1},\eta_{2}) = -2\Ric_{G}(\eta_{1},\eta_{2})+ (D_{\cdot}DG)_{\cdot}(\eta_{1},\eta_{2}) 
-D_{\cdot}G(\eta_{1},\cdot_{1})D_{\cdot}G(\cdot_{1},\eta_{2})
 \\ \nonumber & \ \qquad \qquad \qquad +\frac{1}{2}D_{\cdot}G(\cdot_{1},\cdot_{1})D_{\cdot}G(\eta_{1},\eta_{2})  -\frac{1}{2}G(F(\cdot_{1},\cdot_{2}), \eta_{1})G(F(\cdot_{1},\cdot_{2}), \eta_{2})
\\ \nonumber &\ G(\frac{\del A}{\del t}v, \eta) = D_{v}G([\cdot,\eta],\cdot)  -G(D_{\cdot}F(v,\cdot),\eta) -D_{\cdot}G(F(v,\cdot),\eta) 
 \\  \nonumber & \qquad \qquad \qquad -\frac{1}{2}G(F(v,\cdot_{2}),\eta)D_{\cdot_{2}}G(\cdot_{1},\cdot_{1}) 
\\ \nonumber &\ \frac{\partial g}{ \partial t}(v_{1},v_{2}) = -2\Ric_{g}(v_{1},v_{2}) +(D_{v_{1}}DG)_{v_{2}}(\cdot,\cdot) -\frac{1}{2}D_{v_{1}}G(\cdot_{1},\cdot_{2})D_{v_{2}}G(\cdot_{1},\cdot_{2})
 \\ \nonumber & \qquad \qquad \qquad +G(F(v_{1},\cdot),F(v_{2},\cdot)).
\end{align}

In addition, when analyzing the functionals defined in Definitions \ref{DEFWL} and \ref{DEFW2}, it will be important to consider the following conjugate heat equation for $f(t) \in C^{\infty}(M)$:
\begin{align} \label{EQRC4}
 \frac{\partial}{\partial t}e^{-f}=-\Delta e^{-f} + \Big{(}R_{g} -\frac{1}{4}|DG|^{2} -\frac{1}{2}|F|^{2} +\frac{n}{2t} -\frac{1}{2}D_{\nabla f}G(\cdot,\cdot)\Big{)} e^{-f}.
\end{align}

\subsubsection{Gauged Invariant Ricci Flow Equations}
Given a solution $(\bar{g}(t),f(t))$ to (\ref{EQRICCI}) and (\ref{EQRC4}), let $X_{t} \in \Gamma(TP)$ be the  horizontal lift of  $w_{t}=-\frac{1}{2}g^{-1}DG(\cdot,\cdot) -\nabla f(t)$ with respect to $\bar{A}(t)$.  Pulling back by the flow $\psi_{t}$ generated by $X(t)$, with $\psi|_{t=1}=\mathbb{1}$, we obtain the following gauged Ricci flow equations \cite{GS1}:
\begin{align}
\label{EQGRC1}&\ \frac{\partial G}{\partial t}(\eta_{1},\eta_{2}) = -2\Ric_{G}(\eta_{1},\eta_{2})+ (D_{\cdot}DG)_{\cdot}(\eta_{1},\eta_{2}) 
-D_{\cdot_{2}}G(\eta_{1},\cdot_{1})D_{\cdot_{2}}G(\cdot_{1},\eta_{2})
 \\ \nonumber & \ \qquad \qquad \qquad  -\frac{1}{2}G(F(\cdot_{1},\cdot_{2}), \eta_{1})G(F(\cdot_{1},\cdot_{2}), \eta_{2}) -D_{\nabla f}G (\eta_{1},\eta_{2})
\\ &\ G(\frac{\del A}{\del t}v, \eta) = D_{v}G([\cdot,\eta],\cdot)  -G(D_{\cdot}F(v,\cdot),\eta) -D_{\cdot}G(F(v,\cdot),\eta) 
 \\ \nonumber & \qquad \qquad \qquad -G(F(\nabla f, v) , \eta)
\end{align}
\begin{align}
 &\ \frac{\partial g}{ \partial t}(v_{1},v_{2}) = -2\Ric_{g}(v_{1},v_{2}) +\frac{1}{2}D_{v_{1}}G(\cdot_{1},  
 \cdot_{2})D_{v_{2}}G(\cdot_{1}, \cdot_{2})
 \\ \nonumber & \qquad \qquad \qquad +G(F(v_{1},\cdot),F(v_{2},\cdot)) -2\nabla^{2}f
\\ \label{EQGRC4} & \frac{\partial}{\partial t}e^{-f}=-\Delta e^{-f} + \Big{(} |\nabla f|^{2} +R_{g} -\frac{1}{4}|DG|^{2} -\frac{1}{2}|F|^{2} +\frac{n}{2t} \Big{)} e^{-f}.
\end{align}

The above equations preserve a certain measure:
\begin{lemma}\label{LEMUC}
Given a solution $(\bar{g}(t),f(t))$ to (\ref{EQGRC1})--(\ref{EQGRC4}), one has $\frac{\partial}{\partial t} (u \hspace{.3mm} dV_{g})=0,$ where $u=\frac{e^{-f}}{(4 \pi t)^{\frac{n}{2}}}$. 
\end{lemma}

\section{Lott's Functional in the Nonabelian Setting} \label{SECLOTTF}
In this section, we derive new monotonicity results for Lott's functional for Ricci flows on nilpotent principal bundles.

The $\mathcal{W}_{+}$-functional in \cite{LOTT2} is given as follows: 
\begin{defn}\label{DEFWL} Let $\bar{g}$ be an invariant metric on a $\mathcal{G}$--principal bundle $P\rightarrow M$, which is fibered over a compact, $n$-dimensional  manifold. Also let  $f \in C^{\infty}(M)$ and $\tau \in \mathbb{R}_{>0}$.  Define
\begin{equation*}
\mathcal{\WW}_{L,+}(\bar{g},f, \tau)= \int_{M} \Big{(} \tau (|\nabla f|^{2}+R_{g} -\frac{1}{4} |DG|^{2} -\frac{1}{4} |F|^{2}) -f+n \Big{)} \frac{e^{-f}}{(4\pi \tau)^{\frac{n}{2}}}dV_{g}. 
\end{equation*}
\end{defn}
This functional is scale invariant: $\mathcal{W}_{L,+}(c\bar{g},f,c\tau)=\mathcal{W}_{L,+}(\bar{g},f,\tau)$, for $c>0$. 

Recalling the $\delta$-map as defined in (\ref{DEFDEL}) and Notation \ref{RMKDEL}, we have 

\begin{thm}\label{THMDWL} Let $\bar{g}(t)$ be an invariant Ricci flow solution defined for $t\in (0,\infty)$ on a nilpotent $\mathcal{G}$-principal bundle $P\rightarrow M$, which is fibered over a compact, $n$-dimensional   manifold. Suppose the curvature $F_{\bar{A}(t)}=0$.  Also let $f_{t} \in C^{\infty}(M)$ be a solution to  (\ref{EQRC4}) and define $u_{t}=\frac{e^{-f_{t}}}{(4\pi t)^{\frac{n}{2}}}$.  The following holds true: 
\begin{align}
  \frac{\del }{\del t}\WW_{L,+}(\bar{g}(t),f(t),t) =& \ \frac{t}{2} \int_{M} \big{|} (D_{\cdot}DG)_{\cdot}- D_{\cdot}G(*,\cdot_{1})D_{\cdot}G(*,\cdot_{1}) -D_{\nabla f}G \big{|}^{2} u \hspace{.1mm}dV_{g}
\\ \nonumber  & +  \frac{t}{2} \int_{M} \big{|} 2\Ric_{g} +\frac{g}{t} -\frac{1}{2}DG(\cdot_{1},\cdot_{2})DG(\cdot_{1},\cdot_{2}) +2\nabla^{2}f \big{|}^{2} u \hspace{.1mm}dV_{g}
\\  \nonumber  & +\frac{t}{4} \int_{M} \big{|}\delta(G^{-1}DG)  \big{|}^{2} u \hspace{.1mm}dV_{g}
\\ \nonumber & + t \int_{M} \big{|}DG([*, \cdot],\cdot)\big{|}^{2}u \hspace{.1mm}dV_{g}. 
\end{align}
\end{thm}

\begin{cor}
In the setup of the above theorem, $\mathcal{W}_{L,+}(\bar{g}(t),f(t),t)$ is monotone. 
\end{cor}

\begin{rmk} \label{RMK64}
In \cite{LOTT2}, Lott proved the monotonicity of $\WW_{L,+}$ along an invariant Ricci flow on an abelian $\mathcal{G}$-principal bundle, without the $F=0$ assumption.
\end{rmk}

The main observation in proving Theorem \ref{THMDWL} for the nonabelian setting is the following lemma.

\begin{lemma} \label{LEMREL1} Let $\bar{g}$ be an invariant metric on a $\mathcal{G}$-principal bundle $P \rightarrow M$ and let $f \in C^{\infty}(M)$. We have
\begin{align}
\label{EQRICGD} e^{-f}&\Big{<}\Ric_{G}, (D_{\cdot}DG)_{\cdot}- D_{\cdot_{1}}G(*,\cdot)D_{\cdot_{1}}G(*,\cdot) -D_{\nabla f}G \Big{>}
\\& \nonumber= -\frac{1}{4}e^{-f}\Big{|} \delta(G^{-1}DG) \Big{|}^{2} -d^{*}\Big{(}e^{-f} \big< \Ric_{G},DG\big> \Big{)}.
\end{align}
\end{lemma}

\begin{proof}[Proof of Lemma \ref{LEMREL1}]
Consider
\begin{align*}
d^{*}&\Big{(}e^{-f} \big<\Ric_{G},DG \big> \Big{)} 
\\= &-e^{-f} \big< D\Ric_{G}, DG \big>
 +e^{-f}\big< \Ric_{G}(*,*), D_{\cdot_{1}}G(*, \cdot)D_{\cdot_{1}}G(\cdot,*) \big>
 \\ & -e^{-f} \big< \Ric_{G},  (D_{\cdot}DG)_{\cdot}- D_{\cdot_{1}}G(*, \cdot)D_{\cdot_{1}}G(\cdot,*) -D_{\nabla f}G \big>.
\end{align*}
(\ref{EQRICGD}) follows by applying (\ref{EQDRICS}) to the first two terms in the above expression.
\end{proof}

\begin{proof}[Proof of Theorem \ref{THMDWL}]
We compute $\frac{\partial}{\partial t}\mathcal{W}_{L,+}(\bar{g}(t),f(t),t)$ along a solution of the gauged Ricci flow equations (\ref{EQGRC1})-(\ref{EQGRC4}). We will be using the $F=0$ condition implicitly.

Setting $u=\frac{e^{-f}}{(4\pi t)^{\frac{n}{2}}}$, we have 
\begin{align}
\label{DWL2}\frac{\partial}{\partial t}\mathcal{W}_{L,+}(\bar{g}(t),f(t),t)
    =&  \int_{M}  \frac{\partial}{\partial t} \Big{(} t(|\nabla f|^{2}+ \R_{g}) -f +n \Big{)}  u \hspace{.2mm} dV_{g}
    \\ \nonumber & -\frac{1}{4}  \int_{M}   \frac{\partial}{\partial t} \Big( t|DG|^{2} \Big{)} u \hspace{.2mm} dV_{g}, 
\end{align}
where we used $\frac{\partial}{\partial t} u \hspace{.2mm} dV_{g}=0$ by Lemma \ref{LEMUC}.

Using standard results and those in \cite{GS1}, we have
\begin{align}
\int_{M}  \frac{\partial}{\partial t} \Big{(} t(|\nabla f|^{2}+ & \R_{g}) -f+n \Big{)} u \hspace{.2mm} dV_{g}
  \\ \nonumber  &=\frac{t}{2}\int_{M} \Big{ <} \frac{\partial g}{\partial t} -\frac{g}{t}, -\frac{g}{t} -2\Ric_{g} -2\nabla^{2}f  \Big{ >}   u \hspace{.2mm} dV_{g}
\end{align}
\begin{align}
 -\frac{1}{4}\int_{M}  &  \frac{\partial}{\partial t} \Big{(}t|DG|^{2} \Big{)} u \hspace{.2mm} dV_{g}
  \\ \nonumber = & \ \frac{t}{2} \int_{M} \Big{< }\frac{\partial G}{\partial t}, (D_{\cdot}DG)_{\cdot} -D_{\cdot_{2}}G(*,\cdot_{1})D_{\cdot_{2}}G(*,\cdot_{1}) -D_{\nabla f}G \Big{>}u \hspace{.2mm} dV_{g}
    \\ \nonumber & -t\int_{M} \Big{<}\frac{\partial A}{\partial t}(*), D_{*}G(\cdot_{1},\cdot_{2})G^{-1}G([\  , \cdot_{1}], \cdot_{2})  \Big{ >}   u \hspace{.2mm} dV_{g}
      \\ \nonumber &+ \frac{t}{4}\int_{M} \Big{ <} \frac{\partial g}{\partial t}, DG(\cdot_{1}, \cdot_{2})DG(\cdot_{1},\cdot_{2}) \Big{ >}   u \hspace{.2mm} dV_{g}
\\ \nonumber & -\frac{1}{4}\int_{M} \big| DG \big|^{2} u \hspace{.2mm} dV_{g}.
\end{align}

Inserting these formulas into (\ref{DWL2}) yields
\begin{align*}
\frac{\partial}{\partial t}\mathcal{W}_{L,+}&(\bar{g}(t),f(t),t) 
         \\=& \ \frac{t}{2} \int_{M} \Big{< }\frac{\partial G}{\partial t}, (D_{\cdot}DG)_{\cdot} -D_{\cdot_{2}}G(*,\cdot_{1})D_{\cdot_{2}}G(*,\cdot_{1}) -D_{\nabla f}G \Big{>}u \hspace{.2mm} dV_{g}
    \\& -t\int_{M} \Big{<}\frac{\partial A}{\partial t}(*), D_{*}G(\cdot_{1},\cdot_{2})G^{-1}G([\  , \cdot_{1}], \cdot_{2})  \Big{ >}   u \hspace{.2mm} dV_{g}
      \\& +\frac{t}{2}\int_{M} \Big{ <} \frac{\partial g}{\partial t} -\frac{g}{t}, -\frac{g}{t} -2\Ric_{g} +\frac{1}{2}DG(\cdot_{1}, \cdot_{2})DG(\cdot_{1},\cdot_{2}) -2\nabla^{2}f  \Big{ >}   u \hspace{.2mm} dV_{g}.
 \end{align*}
Using the gauged Ricci flow equations (\ref{EQGRC1})-(\ref{EQGRC4}), gives 
\begin{align*}
\frac{\partial}{\partial t}\mathcal{W}_{L,+}&(\bar{g}(t),f(t),t)
\\  =&\ \frac{t}{2} \int_{M} \Big|  (D_{\cdot}DG)_{\cdot} -D_{\cdot_{2}}G(*,\cdot_{1})D_{\cdot_{2}}G(*,\cdot_{1}) -D_{\nabla f}G    \Big|^{2} u \hspace{.2mm} dV_{g}
\\& -t \int_{M} \Big< \Ric_{G}, (D_{\cdot}DG)_{\cdot} -D_{\cdot_{2}}G(*,\cdot_{1})D_{\cdot_{2}}G(*,\cdot_{1}) -D_{\nabla f}G    \Big> u \hspace{.2mm} dV_{g}
\\&+\frac{t}{2} \int_{M} \Big|   \frac{g}{t} +2\Ric_{g} -\frac{1}{2}DG(\cdot_{1}, \cdot_{2})DG(\cdot_{1},\cdot_{2}) +2\nabla^{2}f \Big|^{2} u \hspace{.2mm} dV_{g}
\\& +t \int_{M} \Big| DG([\ ,\cdot], \cdot) \Big|^{2} u \hspace{.2mm} dV_{g}.
\end{align*} 

Applying Lemma \ref{LEMREL1} to the second term in the above expression, completes the proof of the theorem. 

\end{proof}

\section{Consequences of $\mathcal{W}_{L,+}$  Constant}
In this section, we investigate the condition that $\mathcal{W}_{L,+}(\bar{g}(t),f(t), t)$ is constant along a Ricci flow. This will be important when analyzing blowdown limits in Section \ref{SECBLW1}.
\begin{prop} \label{PROPWLCON}
Let $\bar{g}(t)$ be an invariant Ricci flow solution defined for $t \in (0, \infty)$ on a nilpotent $\mathcal{G}$-principal bundle $P\rightarrow M$, which is fibered over a compact n-dimensional manifold, and let $f_{t} \in C^{\infty}(M)$ satisfy (\ref{EQRC4}). Suppose the curvature $F_{\bar{A}(t)}=0$. If $\frac{\partial}{\partial t}\mathcal{W}_{L,+}(\bar{g}(t), f(t),t)=0$ then the following holds true  
\begin{enumerate}
 \item[1)]  $(D_{\cdot}DG)_{\cdot}- D_{\cdot}G(*,\cdot_{1})D_{\cdot}G(*,\cdot_{1})=0$
 \item[2)]  $\Ric_{g}  +\frac{g}{2t} -\frac{1}{4}DG(\cdot_{1},\cdot_{2})DG(\cdot_{1},\cdot_{2}) =0 $
\item[3)] $\delta(G^{-1}DG)=0$ and $DG(\cdot,\cdot)=0$
\item[4)] $D\Rc_{G}=0$.
  \end{enumerate}
Moreover, we have
\begin{enumerate}
\item[5)] $\frac{\partial A}{\partial t}=0 \text{ and } \frac{\partial D}{\partial t}=0$
\item[6)] $\frac{\del G}{\del t}=-2\Ric_{G}$
\item[7)]$g(t)= tg|_{t=1}$
\item[8)] $d\R_{G}=0$ and $\frac{d}{d t}\R_{G}= 2|\Ric_{G}|^{2}.$
\end{enumerate}
\end{prop}
We now specialize to the case when $\mathcal{G}$ is $Nil^{3}$. As above, define $Z(\mathfrak{G})=\cup_{m \in M} Z(\mathfrak{G}|_{m})$, where $Z(\mathfrak{G}|_{m})$ is the center of the Lie algebra $(\mathfrak{G}|_{m}, [,])$. We have
\begin{cor} \label{CORWLCON}
 Given the setup of  Proposition \ref{PROPWLCON}, if the structure group $\mathcal{G}$ is $Nil^{3}$ then the following conditions also hold true: 
\begin{enumerate}
\item[1)] $\mathfrak{G}_{0}:=Z(\mathfrak{G})^{\perp}$ is independent of  $t$
\item[2)] under the splitting  $\mathfrak{G}=\mathfrak{G}_{0} \oplus Z(\mathfrak{G})$,  $ D= D \oplus D$  
\item[3)] $DG|_{Z(\mathfrak{G})}=0$ and  $tr_{\mathfrak{G}_{0}}DG(\cdot,\cdot)=0$. 
\end{enumerate} 
Moreover,  we have 
\begin{align*}
      &  \quad  a) \ \R_{G}= -\frac{1}{6}\Big( \frac{1}{t+C} \Big),\text{ where } C \geq 0 
        \\& \quad b) \ G(\eta,\eta')=\lp{\frac{t+C}{1+C}}^{\frac{1}{3}}G|_{t=1}        
                           (\eta,\eta'), \text{ for all } \eta,\eta' \in \mathfrak{G}_{0}
        \\& \quad c) \ G(\eta,\eta')=\lp{\frac{t+C}{1+C}}^{-\frac{1}{3}}G|_{t=1}
                           (\eta,\eta'), \text{ for all } \eta,\eta' \in Z(\mathfrak{G}).
\end{align*}
\end{cor}

We will first prove Proposition \ref{PROPWLCON} and then Corollary \ref{CORWLCON}.

\begin{proof}[Proof of Proposition \ref{PROPWLCON}] By Theorem \ref{THMDWL}, the condition $\frac{ \partial}{\partial t}\mathcal{W}_{L,+}(\bar{g}(t),f(t),t)=0$ is equivalent to 
\begin{align}
 \label{EQL1} &(D_{\cdot}DG)_{\cdot} - D_{\cdot}G(*, \cdot_{1})D_{\cdot}G(*,\cdot_{1})-D_{\nabla f}G=0
\\ \label{EQL2} &Ric_{g}+\frac{g}{2t} -\frac{1}{4}DG(\cdot_{1},\cdot_{2})DG(\cdot_{1},\cdot_{2}) +\nabla^{2}f=0
\\ \label{EQL3} & \delta(G^{-1}DG)=0
 \\  \label{EQL4} & DG([*,\cdot],\cdot)=0.
\end{align}

By (\ref{EQDR4}), Lemma \ref{LEMRCDEL1} and (\ref{EQL3}),  
\begin{align} \label{EQDRQ1}
    d \R_{G}= - \big<\Ric_{G}, DG \big>= -\frac{1}{4} \big< \delta(G^{-1}DG), [,] \big> =0.
\end{align}
Also, combining (\ref{EQL3}) with (\ref{EQDRC2}) gives $D\Rc_{G}=0$.

To prove Part 3), using (\ref{EQL1}), consider
\begin{align*}
 d^{*}(e^{-f}DG(\cdot,\cdot))= -e^{-f}\Big{(}(D_{\cdot}DG)_{\cdot}(\cdot_{1},\cdot_{1}) -|DG|^{2} -D_{\nabla f}G(\cdot,\cdot) \Big{)}=0.
\end{align*}
Setting $DG(\cdot,\cdot)=dh$ for some $h\in C^{\infty}(M)$, where we used Lemma \ref{propGRAD}, we have 
\begin{align*}
0= hd^{*}(e^{-f}dh)= d^{*}(e^{-f}hdh)+e^{-f}|dh|^{2}.
\end{align*}
Integrating over $M$, yields $DG(\cdot,\cdot)=dh=0$.

Next, we prove Parts 1) and 2). We will be using Notation \ref{NOTABARR} regarding $\bar{\R}$ and $\bar{\Ric}$ and will implicitly be using the $F=0$ condition. 

Applying  (\ref{EQL1}) and $DG(\cdot,\cdot)=0$ to Lemma \ref{p:scalar}, we obtain
\begin{align}
\nonumber \bar{\R}&= \R_{g}+ \R_{G} -(D_{\cdot}DG)_{\cdot}(\cdot_{1},\cdot_{1}) -\frac{1}{4}D_{\cdot}G(\cdot_{1}, 
\cdot_{1})D_{\cdot}G(\cdot_{2},\cdot_{2}) +\frac{3}{4}|DG|^{2}
\\ \label{EQR1} &=\R_{g}+\R_{G}- \frac{1}{4}|DG|^{2}.
\end{align}
The goal is to derive the following two expressions for $\frac{\partial \bar{\R}}{\partial t}$ to which we will then apply a  minimum principle argument.
\begin{align}
\label{EQDR1} \frac{\partial \bar{\R}}{\partial t} =-\frac{1}{t}(\R_{g}-\frac{1}{4}|DG|^{2}) +2|\Ric_{G}|^{2} + \big< \nabla f, \nabla \bar{\R} \big>
\end{align}
\begin{align}
 \label{EQDR2} \frac{\partial \bar{R}}{\partial t} 
    =& \ \Delta_{g} \bar{\R} + 2|\Ric_{G}|^{2} - \frac{n}{2t^{2}} -\frac{2}{t}(\R_{g} -\frac{1}{4}|DG|^{2})
 \\& \nonumber + 2|\Ric_{G}- \bar{\Ric}_{ij}|^{2}
     +2|\bar{\Ric}_{ab}+\frac{g}{2t}|^{2}.
\end{align}


To derive (\ref{EQDR1}), using (\ref{EQL1})-(\ref{EQL4}) and $DG(\cdot,\cdot)=0$, the Ricci flow equations (\ref{EQRICCI}) become 
\begin{align}
\nonumber \frac{\partial G}{\partial t}&= -2 \Ric_{G}+ D_{\nabla f}G
\\ \label{EQRICC2} \frac{\partial A}{\partial t}&=0 
\\ \nonumber \frac{\partial g}{\partial t}&= \frac{g}{t} +2 \nabla^{2}f.
\end{align} 
Now consider (see for example \cite{GS1})
\begin{align}
 \nonumber \frac{\partial}{\partial t} |DG|^{2}
  =& -2 \big{<}\frac{\partial G}{\partial t}, (D_{\cdot}DG)_{\cdot}-D_{\cdot}G(*,\cdot_{1})D_{\cdot}G(*,\cdot_{1})\big{>}
  \\ & +4 \big<\frac{\partial A}{\partial t}(*), D_{*}G(\cdot_{1},\cdot_{2})G^{-1}G([\hspace{2mm}, \cdot_{1}],\cdot_{2})\big>    
 \\ \nonumber &-\big{<}\frac{\partial g}{\partial t}, DG(\cdot_{1},\cdot_{2})DG(\cdot_{1},\cdot_{2})   \big{>}
  -2d^{*} \big< DG, \frac{\partial G}{\partial t} \big>. 
\end{align}
 Plugging in (\ref{EQRICC2}), gives 
\begin{align}
 \nonumber \frac{\partial}{\partial t} |DG|^{2}
   &= -2 \big{<}-2\Ric_{G}, (D_{\cdot}DG)_{\cdot}-D_{\cdot}G(*,\cdot_{1})D_{\cdot}G(*,\cdot_{1})\big{>}
 \\ \nonumber & \ \ \ - \big{<}\frac{g}{t}, DG(\cdot_{1},\cdot_{2})DG(\cdot_{1},\cdot_{2})   \big{>}
                                  -2d^{*}\big<DG, -2\Ric_{G} \big> +\big<\nabla f, \nabla |DG|^{2} \big>
 \\ \nonumber & = -|\delta(G^{-1}DG)|^{2} -\frac{1}{t}|DG|^{2} + \big<\nabla f, \nabla |DG|^{2} \big>
\\ \label{EQDG2} &  =  -\frac{1}{t}|DG|^{2} + \big< \nabla f, \nabla |DG|^{2} \big>,
\end{align}
where we used Lemma \ref{LEMREL1} with $f=0$  to obtain the second equality, and (\ref{EQL3}) to obtain the third.

(\ref{EQDR1}) then follows from (\ref{EQR1}),  (\ref{EQRICC2}) and (\ref{EQDG2}).

To derive (\ref{EQDR2}), since $\bar{g}(t)$ is a Ricci flow solution,
\begin{align} \label{EQDR3}
\frac{\partial \bar{R}}{\partial t}&= \Delta_{\bar{g}}\hspace{.4mm} \bar{R}+2|\bar{Ric}|^{2}.  
\end{align}
We now examine the above two terms. We first claim
\begin{align}
\label{EQLAP1} \Delta_{\bar{g}}\hspace{.4mm} \bar{R}= \Delta_{g} \hspace{.4mm} \bar{R}. 
\end{align}
To prove this, choose orthonormal frames $\{\eta_{i}\}$ and $\{v_{a}\}$ for $\mathfrak{G}$ and $TM$ respectively over an open set $U \subset M$. We will be using the  Lie algebroid connection $\nabla$ of Definition  \ref{DEFNAB} and the exterior derivative $d_{\mathcal{E}}$ of Definition \ref{DEFD1}.  We have 
\begin{align*}
\Delta_{\bar{g}}\hspace{.4mm} \bar{\R} &= (\nabla_{\eta_{i}}d_{\mathcal{E}}\bar{\R})(\eta_{i}) +  (\nabla_{v_{a}}d_{\mathcal{E}}\bar{\R})(v_{a}) 
    \\&= -d_{\mathcal{E}}\bar{\R}(\nabla_{\eta_{i}}\eta_{i}) +v_{a}[d\bar{\R}(v_{a})]
             -d_{\mathcal{E}}\bar{\R}(\nabla_{v_{a}}v_{a})
    \\&= -d\bar{\R} \big( -\frac{1}{2}g^{-1}DG(\cdot,\cdot) \big) +v_{a}[d\bar{\R}(v_{a})]
             -d\bar{\R}(\nabla^{g}_{v_{a}}v_{a})
    \\&= \Delta_{g}\bar{\R},
\end{align*}
where we used Lemma \ref{LEMNABDEC} and $DG(\cdot,\cdot)=0$.

For the second term in (\ref{EQDR3}), using Lemma \ref{p:Ricci} and (\ref{EQL1})-(\ref{EQL4}), we have
\begin{align}
&\nonumber \bar{\Ric}_{ij}= \Ric_{G}-\frac{1}{2}(D_{\cdot}DG)_{\cdot} +\frac{1}{2}D_{\cdot}G(*, \cdot_{1})D_{\cdot}G(*,\cdot_{1})
\\ \label{EQRICL1}&\bar{\Ric}_{ia}= 0
\\ \nonumber & \bar{\Ric}_{ab}=\Ric_{g}-\frac{1}{4} DG(\cdot_{1},\cdot_{2})DG(\cdot_{1},\cdot_{2}).
\end{align} 
Then 
\begin{align}
\nonumber 2|\bar{\Ric}|^{2}
    &=2|\bar{\Ric}_{ij}|^{2} +2 |\bar{\Ric}_{ab}|^{2}
    \\ \label{EQRCBAR1} &=   2|\Ric_{G}|^{2} - \frac{n}{2t^{2}} -\frac{2}{t}(\R_{g} -\frac{1}{4}|DG|^{2})
 \\ \nonumber & \ \ \  + 2|\Ric_{G}- \bar{\Ric}_{ij}|^{2}
          +2|\bar{\Ric}_{ab}+\frac{g}{2t}|^{2}.
\end{align}
In the last equality, we used $\big< \Ric_{G}, \Ric_{G}- \bar{\Ric}_{ij} \big>=0$. One way to show this is to apply  (\ref{EQRICL1}), Lemmas \ref{LEMREL1} and \ref{LEMRCDEL1}, and (\ref{EQL3}) to obtain 
\begin{align*}
2\big< \Ric_{G}, \Ric_{G}- \bar{\Ric}_{ij} \big> &= \big< \Ric_{G}, (D_{\cdot}DG)_{\cdot} - D_{\cdot}G(*, \cdot_{1})D_{\cdot}G(*,\cdot_{1}) \big>
\\& = -\frac{1}{4}|\delta(G^{-1}DG)|^{2}- \frac{1}{4} d^{*} \big<\delta(G^{-1}DG), [,] \big> =0.
\end{align*}
Alternatively, one may use (\ref{EQL1}) instead of Lemma \ref{LEMREL1}. Combining (\ref{EQDR3}),  (\ref{EQLAP1}) and (\ref{EQRCBAR1}) proves (\ref{EQDR2}).

If we now equate (\ref{EQDR1}) and (\ref{EQDR2}) and use $d\R_{G}=0$ as in (\ref{EQDRQ1}), we obtain 
\begin{align} 
\nonumber \big< \nabla f, \nabla \big(\R_{g}-\frac{1}{4}|DG|^{2} +&\frac{n}{2t} \big) \big>
\\ \label{EQMIN1}  =& \ \Delta_{g} \big( \R_{g}-\frac{1}{4}|DG|^{2} +\frac{n}{2t} \big) - \frac{1}{t} \big(\R_{g}-\frac{1}{4}|DG|^{2} +\frac{n}{2t} \big)
\\ \nonumber & + 2|\Ric_{G}- \bar{\Ric}_{ij}|^{2}
  +2|\bar{\Ric}_{ab}+\frac{g}{2t}|^{2}.
\end{align}

At the same time, by (\ref{EQL2}), we have 
\begin{align*}
\int_{M} (R_{g}-\frac{1}{4}|DG|^{2} +&\frac{n}{2t}) dV_{g}= -\int_{M} \Delta f dV_{g}=0.
\end{align*}
Hence $\big( R_{g}-\frac{1}{4}|DG|^{2} +\frac{n}{2t} \big) _{min}<0$ or $R_{g}-\frac{1}{4}|DG|^{2} +\frac{n}{2t}=0$. Since the former case would contradict (\ref{EQMIN1}), we have that $R_{g}-\frac{1}{4}|DG|^{2} +\frac{n}{2t}=0$. (\ref{EQMIN1}) then implies that $Ric_{G}- \bar{Ric}_{ij}=0$ and $\bar{Ric}_{ab}+\frac{g}{2t}=0$. Using (\ref{EQRICL1}), completes the proof of Parts 1) and 2) of the proposition. 

The proof of the rest of the proposition is straightforward. 
\end{proof}

\begin{rmk}
The proof of Parts 1) and 2) of Proposition \ref{PROPWLCON} generalizes that of Proposition 4.67 in \cite{LOTT2} to the nonabelian setting.
\end{rmk}
\begin{proof}[Proof of Corollary \ref{CORWLCON}]

Part 1) follows from the fact that, by Proposition \ref{PROPWLCON} and (\ref{EQRCNIL}), $G(t)$ satisfies $\frac{\partial G}{\partial t}= -2\Ric_{G}$ and
\begin{align} \label{EQNILSPLIT2}
\Ric_{G}= \R_{G}G \oplus -\R_{G}G,
\end{align}
based on the splitting $\mathfrak{G}= Z(\mathfrak{G})^{\perp}\oplus Z(\mathfrak{G})$.

For the proof of the rest of the corollary, we will be using the  $\delta(G^{-1}DG)=0$ condition in Proposition \ref{PROPWLCON}. By Lemma \ref{LEMDERNIL}, this is equivalent to 
\begin{align}
 \label{EQDGSPL} DG&=DG \oplus DG
\\ \label{EQDGSPL2} DG(\cdot,\cdot)&= 2tr_{Z(\mathfrak{G})}DG(\cdot,\cdot),
\end{align}
where (\ref{EQDGSPL}) is based on the splitting of $\mathfrak{G}= Z(\mathfrak{G})^{\perp} \oplus Z(\mathfrak{G})$. 

To  prove Part 2), first note by Lemma \ref{LEMDZ}, $D$ preserves sections of $Z(\mathfrak{G})$. Now let $\eta$ and $\eta'$ be sections of $Z(\mathfrak{G})^{\perp}$ and  $Z(\mathfrak{G})$, respectively, and consider 
    \begin{align*}
    0&= D(G(\eta,\eta') ) \\ &= DG(\eta, \eta') +G(D\eta,\eta') +G(\eta,D\eta') \\ & =G(D\eta,\eta'),
    \end{align*}
    where we used $DG(\eta, \eta')=0$, which follows from (\ref{EQDGSPL}), and $G(\eta,D\eta')=0$. Hence $D\eta$ is a section of $Z(\mathfrak{G})^{\perp}$.

To prove Part 3), combining Part 3) of Proposition \ref{PROPWLCON} with (\ref{EQDGSPL2}) gives 
\begin{align}
0=tr_{\mathfrak{G}_{0}}DG(\cdot,\cdot) + tr_{Z(\mathfrak{G})}DG(\cdot,\cdot)= 2tr_{Z(\mathfrak{G})}DG(\cdot,\cdot),
\end{align}
where $\mathfrak{G}_{0}:= Z(\mathfrak{G})^{\perp}$. Hence, $tr_{\mathfrak{G}_{0}}DG(\cdot,\cdot)$ and $tr_{Z(\mathfrak{G})}DG(\cdot,\cdot)$ are both zero. Since $Z(\mathfrak{G})$ has rank one, we also obtain $DG|_{Z(\mathfrak{G})}=0.$

Lastly, Parts a)-c) follow from solving Parts 6) and 8) of Proposition \ref{PROPWLCON} together with (\ref{EQNILSPLIT2}).  
\end{proof}

\section{Blowdown Limits and the $\mathcal{W}_{L,+}$-Functional} \label{SECBLW1}

Using Theorem \ref{THMDWL} and Proposition \ref{PROPWLCON}, we obtain the following results about blowdown limits on nilpotent principal bundles.
\begin{thm}\label{THMBLOW}
Let $\bar{g}(t)$ be an invariant Ricci flow solution defined for  $t \in (0, \infty)$   on a nilpotent $\mathcal{G}$-principal bundle $P\rightarrow M$, over a compact n-dimensional manifold, such that the curvature $F_{\bar{A}(t)}=0$. Define $\bar{g}_{i}(t)=s_{i}^{-1}\bar{g}(s_{i}t)$, where $\{s_{i}\}$ is a sequence of positive numbers such that $\lim s_{i}=\infty$. We suppose there exists a $\mathcal{G}$--principal bundle $P_{\infty} \rightarrow M_{\infty}$, over a compact manifold, together with bundle isomorphisms $\psi_{i}:P_{\infty} \rightarrow P$  such that $\psi_{i}^{*}\bar{g}_{i}(t)$  converges  smoothly and uniformly to $\bar{g}_{\infty}(t)$ on $P_{\infty} \times [j^{-1},j]$, for all $j>0$. The following holds true in the limit: 
\begin{enumerate}
 \item[1)]  $(D_{\cdot}DG)_{\cdot}- D_{\cdot}G(*,\cdot_{1})D_{\cdot}G(*,\cdot_{1})=0$
 \item[2)]  $Ric_{g}  +\frac{g}{2t} -\frac{1}{4}DG(\cdot_{1},\cdot_{2})DG(\cdot_{1},\cdot_{2}) =0 $
\item[3)] $\delta(G^{-1}DG)=0$ and $DG(\cdot,\cdot)=0$
\item[4)] $D\Rc_{G}=0$.
  \end{enumerate}
Moreover, we have
\begin{enumerate}
\item[5)] $\frac{\partial A}{\partial t}=0 \text{ and } \frac{\partial D}{\partial t}=0$
\item[6)] $\frac{\del G}{\del t}=-2Ric_{G}$
\item[7)]$g(t)= tg|_{t=1}$
\item[8)] $d \R_{G}=0$ and $\frac{d}{d t} \R_{G}= 2|\Ric_{G}|^{2}.$
\end{enumerate}
In the case when $\mathcal{G}$ is $Nil^{3}$, the conclusions of Corollary \ref{CORWLCON} also hold true in the limit.
\end{thm}

\begin{rmk}
In \cite{GS3}, we derived Theorem \ref{THMBLOW} in the case when $\mathcal{G}$ is $Nil^{3}$ and $dim M=1$ by constructing a functional, $\mathcal{I}$, specifically for this setting. 
\end{rmk}

\begin{proof} Consider the setup of the theorem and define $\phi_{i}: M_{\infty} \rightarrow M$ to be the diffeomorphisms induced from the bundle maps $\psi_{i}: P_{\infty} \rightarrow P$. Similar to the proof of Proposition 4.79 in \cite{LOTT2}, one may construct a solution $f_{t} \in C^{\infty}(M)$ to the conjugate heat equation (\ref{EQRC4}) with respect to $\bar{g}(t)$ and for $t \in (0, \infty)$. Moreover, if we set $f_{i}(t)=f(s_{i}t)$ then there is a subsequence of $\phi_{i}^{*}f_{i}$ that will converge smoothly and  uniformly on compact subsets of $(0, \infty)$ to $f_{\infty}(t)$ on $M_{\infty}$. $f_{\infty}(t)$ will  satisfy the conjugate heat equation (\ref{EQRC4}) with respect to $\bar{g}_{\infty}(t)$.

Using the properties of $\mathcal{W}_{L,+}$, we then have 
\begin{align*}
\mathcal{W}_{L,+}(\bar{g}_{\infty}(t),f_{\infty}(t),t)
  		&=\lim_{i\rightarrow \infty}\mathcal{W}_{L,+}(\psi_{i}^{*}\bar{g}_{i}(t),\phi_{i}^{*}f_{i}(t),t)
						\\&=	\lim_{i\rightarrow \infty}\mathcal{W}_{L,+}(\bar{g}_{i}(t), f_{i}(t),t)
	\\&= \lim_{i\rightarrow \infty}\mathcal{W}_{L,+}(s_{i}^{-1}\bar{g}(s_{i}t), f(s_{i}t), t)
	   \\&= \lim_{i\rightarrow \infty}\mathcal{W}_{L,+}(\bar{g}(s_{i}t),f(s_{i}t),s_{i}t)
	       \\& = \lim_{ u\rightarrow \infty}\mathcal{W}_{L,+}(\bar{g}(u), f(u),u).
						\end{align*}
	The existence of the last limit follows from the existence of the second to last limit and the fact that $\mathcal{W}_{L,+}(\bar{g}(t), f(t), t)$ is a nondecreasing function of $t$, which follows from Theorem \ref{THMDWL}.

Hence $\mathcal{W}_{L,+}(\bar{g}_{\infty}(t),f_{\infty}(t),t)$ is constant in $t$. Applying Proposition \ref{PROPWLCON} and Corollary \ref{CORWLCON}, completes the proof of the theorem.
\end{proof}

\section{The $\mathcal{W}_{+}$-Functional for Nilpotent Principal Bundles}\label{SECW3}
In this section, we introduce the $\mathcal{W}_{+}$-functional for Ricci flows on nilpotent principal bundles and determine some of its monotonicity properties. In Section \ref{SECBLOWSOL}, we apply the functional to derive blowdown limit results which strengthen those in Theorem \ref{THMBLOW}. In particular, we demonstrate that the limit is locally an expanding Ricci soliton under certain conditions. In Section \ref{SECBLOWSOL2}, we classify the soliton when the structure group of the principal bundle is $Nil^{3}$ and the base manifold is one dimensional.

\subsection{The $\mathcal{W}_{+}$-Functional}
The following functional is a combination of $\mathcal{W}_{L,+}$ of Definition \ref{DEFWL} and $\mathcal{W}_{+}$ of Definition \ref{DEFWLIE} for Lie groups. We will be using the same notation and definitions of Section \ref{SECRICCBAC}.  

\begin{defn}\label{DEFW2} Let $\bar{g}$ be an invariant metric on a $\mathcal{G}$--principal bundle $P\rightarrow M$, which is fibered over a compact, $n$-dimensional  manifold. Also let  $f \in C^{\infty}(M)$, $a \in \mathbb{R}$ and $\tau \in \mathbb{R}_{>0}$.  Define
\begin{align*}
 &\mathcal{W}_{+}(\bar{g},f, a, \tau) \\ &= \mathcal{W}_{L,+}(\bar{g},f, \tau) -a \int_{M} \Big{(}\tau R_{G}+ \tau^{2}|Ric_{G}|^{2}\Big{)} \frac{e^{-f}}{(4\pi \tau)^{\frac{n}{2}}}dV_{g} 
\\&=\int_{M} \Big{(} \tau (|\nabla f|^{2}+R_{g} -\frac{1}{4} |DG|^{2} -\frac{1}{4} |F|^{2} -a R_{G} - a \tau |Ric_{G}|^{2}) -f+n \Big{)} \frac{e^{-f}}{(4\pi \tau)^{\frac{n}{2}}}dV_{g}. 
\end{align*}
\end{defn}
This functional is scale invariant: $\mathcal{W}_{+}(c\bar{g},f,a, c\tau)=\mathcal{W}_{+}(\bar{g},f,a,\tau)$, for $c>0.$

We have the following monotonicity  result for Ricci flows on $Nil^{3}$-bundles:
\begin{thm} \label{THMDW2} Let $\bar{g}(t)$ be an invariant Ricci flow  solution defined for  $t \in (0, \infty)$     on a $Nil^{3}$-principal bundle $P\rightarrow M$, which is fibered over a compact, $n$-dimensional manifold. Suppose the curvature $F_{\bar{A}(t)}=0$. Also let $f_{t} \in C^{\infty}(M)$ be a solution to (\ref{EQRC4}), $u_{t}:=\frac{e^{-f_{t}}}{(4\pi t)^{\frac{n}{2}}}$ and $a \in \mathbb{R}$. The following holds true:
\begin{align*}
\frac{\del }{\del t}&\WW_{+}(\bar{g}(t),f(t), a, t) 
\\=& \ \frac{t}{2} \int_{M} \big{|} (D_{\cdot}DG)_{\cdot}- D_{\cdot}G(*,\cdot_{1})D_{\cdot}G(*,\cdot_{1}) -D_{\nabla f}G \big{|}^{2} u \hspace{.1mm} dV_{g}
\\ & +  \frac{t}{2} \int_{M} \big{|} 2Ric_{g} +\frac{g}{t} -\frac{1}{2}DG(\cdot_{1},\cdot_{2})DG(\cdot_{1},\cdot_{2}) +2\nabla^{2}f \big{|}^{2} u \hspace{.1mm} dV_{g}
\\& +\frac{t}{4}(1-a) \int_{M} \big{|}\delta(G^{-1}DG)  \big{|}^{2} u \hspace{.1mm} dV_{g}
 + t \int_{M} \big{|}DG([*, \cdot],\cdot)\big{|}^{2} u \hspace{.1mm} dV_{g}
\\ &+6at^{2} \int_{M} \big< \Ric_{G}, D_{\cdot}G\big> \big< \Ric_{G}, D_{\cdot}G\big> u \hspace{.2mm} dV_{g}
+\frac{3}{8}at^{2}\int_{M} |[,]|^{2} |\delta(G^{-1}DG)|^{2} u \hspace{.1mm} dV_{g}
\\& +at^{2} \int_{M} |\delta(\Rc_{G} +\frac{\mathbb{1}}{2t})|^{2} u \hspace{.1mm} dV_{g}. 
\end{align*}
\end{thm}

\begin{cor}\label{CORWM2}
In the above setup, $\WW_{+}(\bar{g}(t), f(t), a, t)$ is monotone for $0 \leq a\leq 1$. 
\end{cor}

\begin{rmk}
The conclusion of Theorem \ref{THMDW2} is also true when the structure group of $P$ is abelian since $\WW_{+}$ would  reduce to Lott's functional, $\WW_{L,+}$ \cite{LOTT2}. (See also Theorem \ref{THMDWL} and Remark \ref{RMK64}.)
\end{rmk}

The proof of Theorem \ref{THMDW2} will be based on the following two lemmas.

\begin{lemma} \label{LEMT1} Let $\bar{g}(t)$ be an invariant Ricci flow solution    defined for  $t \in (0, \infty)$       on a nilpotent $\mathcal{G}$-principal bundle $P\rightarrow M$, which is fibered over a compact, $n$-dimensional manifold. Suppose the curvature $F_{\bar{A}(t)}=0$. Also let $f_{t} \in  C^{\infty}(M)$ be a solution to (\ref{EQRC4}) and $u_{t}:=\frac{e^{-f_{t}}}{(4\pi t)^{\frac{n}{2}}}$. The following holds true: 
\begin{align*}
\frac{\partial}{\partial t} \int_{M} t \R_{G} u \hspace{.2mm} dV_{g}
&=  \int_{M} \Big( \frac{t}{4} \big| \delta(G^{-1}DG) \big|^{2}        %
  +2t | \Ric_{G}   |^{2} 
+ \R_{G} \Big) u \hspace{.2mm} dV_{g}.
\end{align*}
\end{lemma}

\begin{proof} In the gauged Ricci flow equations (\ref{EQGRC1})-(\ref{EQGRC4}), we have  
\begin{align*}
 \frac{\partial}{\partial t}& \int_{M}  t \R_{G} \hspace{.6mm} u \hspace{.2mm} dV_{g}
\\&=\int_{M} \Big{(} \R_{G} -t \big< \Ric_{G}, \frac{\partial G}{\partial t} \big> \Big) u \hspace{.2mm} dV_{g}
\\&= -t\int_{M} \big<\Ric_{G},  (D_{\cdot}DG)_{\cdot} -D_{\cdot}G(*,\cdot_{1})D_{\cdot}G(*,\cdot_{1}) -  
        D_{\nabla f}G \big> u \hspace{.2mm} dV_{g}
 \\& \quad +2t \int_{M} \big|\Ric_{G}\big|^{2} u \hspace{.2mm} dV_{g}
 + \int_{M} \R_{G} \hspace{.6mm} u \hspace{.2mm} dV_{g},
\end{align*}
where we used (\ref{EQDR}) and $\frac{\partial}{ \partial t}u \hspace{.2mm} dV_{g}=0$.

The lemma then follows by applying Lemma \ref{LEMREL1} to the above expression. 
\end{proof}

\begin{lemma} \label{LEMT2} Given the setup of Theorem \ref{THMDW2}, we have 
\begin{align*}
&\frac{\partial}{\partial t} \int_{M}  t^{2} \big| Ric_{G}\big|^{2} u \hspace{.2mm} dV_{g}
\\ &=-6t^{2} \int_{M} \big< \Ric_{G}, D_{\cdot}G\big> \big< \Ric_{G}, D_{\cdot}G\big> u \hspace{.2mm} dV_{g}
  -\frac{3}{8}t^{2} \int_{M} \big| [,] \big|^{2} \big| \delta(G^{-1}DG) \big|^{2} u \hspace{.2mm} dV_{g}
\\& \quad -t^{2} \int_{M} \big|  \delta(\Rc_{G})   \big|^{2} u \hspace{.2mm} dV_{g}
 +2t \int_{M} \big| Ric_{G}\big|^{2}u \hspace{.2mm} dV_{g}.
\end{align*}
\end{lemma}

\begin{proof}
Since, by (\ref{EQRCNIL2}), $\big| \Ric_{G}\big|^{2}= 3\R_{G}^{2}$ for $Nil^{3}$, we consider 
\begin{align*}
\frac{\partial}{\partial t} \int_{M} 3t^{2} \R_{G}^{2}u \hspace{.2mm} dV_{g}
= -\int_{M} 6t^{2} R_{G}  \big< Ric_{G}, \frac{\partial G}{\partial t} \big> u \hspace{.2mm} 
        dV_{g}
+\int_{M} 2t \big| \Ric_{G} \big|^{2}  u \hspace{.2mm}dV_{g},
\end{align*}
where we used (\ref{EQDR}) and $\frac{\partial}{\partial t}u \hspace{.2mm}dV_{g}=0.$

Applying (\ref{EQGRC1})-(\ref{EQGRC4}), we obtain
\begin{align}
 \nonumber  \  \frac{\partial}{\partial t} \int_{M} t^{2} |\Ric_{G}|^{2} u \hspace{.2mm} dV_{g}&
  \\ \label{EQT1} = -\int_{M} 6t^{2} &\Big< \R_{G} \Ric_{G}, (D_{\cdot}DG)_{\cdot} -D_{\cdot}G(*,\cdot_{1})D_{\cdot}
                   G(*,\cdot_{1}) - D_{\nabla f}G   \Big>u \hspace{.2mm}dV_{g}
 \\ \label{EQT3}+\int_{M} 12t^{2} &\R_{G}|\Ric_{G}\big|^{2}u \hspace{.2mm}dV_{g}
+\int_{M} 2t \big| \Ric_{G} \big|^{2}  u \hspace{.2mm}dV_{g}.
\end{align}

To analyze (\ref{EQT1}), consider
%
\begin{align*}
e^{-f}& \Big< \R_{G} \Ric_{G},  (D_{\cdot}DG)_{\cdot} -D_{\cdot}G(*,\cdot_{1})D_{\cdot}
                   G(*,\cdot_{1}) - D_{\nabla f}G   \Big>
\\&= -e^{-f} \big< D\big( \R_{G} \Ric_{G} \big), DG \big>
 +e^{-f} \R_{G} \Ric_{G}(\cdot_{1},\cdot_{2})D_{\cdot}G(\cdot_{1},\cdot_{3})D_{\cdot}G(\cdot_{3},
         \cdot_{2})
\\& \quad -d^{*}\Big( e^{-f} \R_{G} \big< \Ric_{G}, DG \big>  \Big)
\\&= -e^{-f} \big( D_{\cdot} \R_{G} \big) \big< \Ric_{G}, D_{\cdot}G \big>
-e^{-f} \R_{G} \big< D \Ric_{G}, DG \big> 
 \\& \quad+e^{-f} \R_{G} \Ric_{G}(\cdot_{1},\cdot_{2})D_{\cdot}G(\cdot_{1},\cdot_{3})D_{\cdot}G(\cdot_{3},
         \cdot_{2})
-d^{*}\Big( e^{-f} \R_{G} \big< \Ric_{G}, DG \big>  \Big).
\end{align*}
Using  (\ref{EQDR4}) and (\ref{EQDRICS}), yields
\begin{align}
\nonumber e^{-f}\Big< &\R_{G} \Ric_{G},  (D_{\cdot}DG)_{\cdot} -D_{\cdot}G(*,\cdot_{1})D_{\cdot}
                   G(*,\cdot_{1}) - D_{\nabla f}G   \Big>
 \\&=e^{-f} \big< \Ric_{G},D_{\cdot}G \big> \big< \Ric_{G},D_{\cdot}G \big>
-\frac{1}{4}e^{-f} \R_{G} \big| \delta(G^{-1}DG)   \big|^{2}
\\ \nonumber & \ \ \ -d^{*}\Big( e^{-f} \R_{G} \big< \Ric_{G}, DG \big>  \Big).
\end{align}
As for (\ref{EQT3}), it follows from Lemma \ref{LEMNILDEL} and (\ref{EQRCNIL2}) that 
\begin{align}
|\delta(\Rc_{G})|^{2}=-12\R_{G}|\Ric_{G}|^{2}.
\end{align}
Combining these results, proves the lemma. 
\end{proof}

\begin{proof}[Proof of Theorem \ref{THMDW2}]
By Lemmas \ref{LEMT1} and \ref{LEMT2},
\begin{align*}
\frac{\del }{\del t}&\WW_{+}(\bar{g}(t),f(t), a, t)
\\ =& \ \frac{\del }{\del t}\WW_{L,+}(\bar{g}(t),f(t),t)
 -\frac{at}{4} \int_{M}  \big| \delta(G^{-1}DG) \big|^{2} u \hspace{.2mm} dV_{g}
\\ &+6at^{2} \int_{M} \big< \Ric_{G}, D_{\cdot}G\big> \big< \Ric_{G}, D_{\cdot}G\big> u \hspace{.2mm} dV_{g}
  +\frac{3a}{8}t^{2} \int_{M} \big| [,] \big|^{2} \big| \delta(G^{-1}DG) \big|^{2} u \hspace{.2mm} dV_{g}
\\ \quad &-a\int_{M} \Big( -t^{2} \big| \delta(\Rc_{G})\big|^{2} + 4t  \big| \Ric_{G}\big|^{2} + \R_{G}  \Big) u \hspace{.2mm} dV_{g}.
\end{align*}
The theorem then follows from Theorem \ref{THMDWL} and the relation: 
\begin{align*}
-t^{2} \big| \delta(\Rc_{G})\big|^{2} + 4t  \big| \Ric_{G}\big|^{2} + \R_{G}
 =-t^{2}|\delta(\Rc_{G}+ \frac{\mathbb{1}}{2t})|^{2}.
\end{align*}
\end{proof}

\section{Blowdown Limits and Solitons} \label{SECBLOWSOL}

In this section, we use the $\mathcal{W}_{+}$-functional to derive blowdown limit rigidity results which are stronger than those in  Theorem \ref{THMBLOW}, in the case when the structure group is $Nil^{3}$. Consequently, we prove that the blowdown limit is locally an expanding Ricci soliton.

To obtain the rigidity results, as above we define $Z(\mathfrak{G})=\cup_{m \in M} Z(\mathfrak{G}|_{m})$, where $Z(\mathfrak{G}|_{m})$ is the center of the Lie algebra $(\mathfrak{G}|_{m}, [,])$. We then have

\begin{thm}\label{THMBLOW2}
Let $\bar{g}(t)$ be an invariant Ricci flow solution   defined for  $t \in (0, \infty)$       on a $Nil^{3}$-principal bundle $P\rightarrow M$, over a compact n-dimensional manifold, such that the curvature $F_{\bar{A}(t)}=0$.  Define $\bar{g}_{i}(t)=s_{i}^{-1}\bar{g}(s_{i}t)$, where $\{s_{i}\}$ is a sequence of positive numbers such that $\lim s_{i}=\infty$. We suppose there exists a $Nil^{3}$--principal bundle $P_{\infty} \rightarrow M_{\infty}$, over a compact manifold, together with bundle isomorphisms $\psi_{i}:P_{\infty} \rightarrow P$  such that $\psi_{i}^{*}\bar{g}_{i}(t)$  converges smoothly and uniformly to $\bar{g}_{\infty}(t)$ on $P_{\infty} \times [j^{-1},j]$, for all $j>0$. The following holds true in the limit: 
\begin{enumerate}
\item[1)] $\mathfrak{G}_{0}:=Z(\mathfrak{G})^{\perp}$ is independent of  $t$
\item[2)] under the splitting  $\mathfrak{G}=\mathfrak{G}_{0} \oplus Z(\mathfrak{G})$,  $ D= D \oplus D$  
\item[3)] $DG|_{Z(\mathfrak{G})}=0$ and  $tr_{\mathfrak{G}_{0}}DG(\cdot,\cdot)=0$ 
 \item[4)]  $(D_{\cdot}DG)_{\cdot}- D_{\cdot}G(*,\cdot_{1})D_{\cdot}G(*,\cdot_{1})=0$
 \item[5)]  $\Ric_{g}  +\frac{g}{2t} -\frac{1}{4}DG(\cdot_{1},\cdot_{2})DG(\cdot_{1},\cdot_{2}) =0 $
\item[6)] $\delta \big( \Rc_{G}+ \frac{\mathbb{1}}{2t} \big)=0$
\item[7)] $\frac{\partial A}{\partial t}=0 \text{ and } \frac{\partial D}{\partial t}=0$
\item[8)]$g(t)= tg|_{t=1}$.
\end{enumerate}
Moreover, in the limit  we have 
\begin{align*}
      &  \quad  a) \ \R_{G}= -\frac{1}{6t}
        \\& \quad b) \ G(\eta,\eta')=t^{\frac{1}{3}}G|_{t=1}        
                           (\eta,\eta'), \text{ for all } \eta,\eta' \in \mathfrak{G}_{0}
        \\& \quad c) \ G(\eta,\eta')=t^{-\frac{1}{3}}G|_{t=1}
                           (\eta,\eta'), \text{ for all } \eta,\eta' \in Z(\mathfrak{G}).
\end{align*}
\end{thm}
\begin{rmk} \label{RMKBS1}
By Lemma \ref{PROPDELNIL}, the condition $\delta(\Rc_{G}+ \frac{\mathbb{1}}{2t})=0$ is equivalent to  $\R_{G}= -\frac{1}{6t}$. Also, $\delta(G^{-1}DG)=0$ follows from Parts 2) and 3) of the theorem and Lemma \ref{LEMDERNIL}. 
\end{rmk}
\begin{proof}
Similar to the proof of Theorem \ref{THMBLOW}, the blowdown limit satisfies 
\begin{align} \label{EQBS1}
\frac{\partial}{\partial t} \mathcal{W}_{+}(\bar{g}_{\infty}(t), f_{\infty}(t), a, t) =0,
\end{align}
 where $a \in [0,1]$ and $f_{\infty}(t) \in C^{\infty}(M_{\infty})$ is a solution to the conjugate heat equation (\ref{EQRC4}) with respect to  $\bar{g}_{\infty}(t)$.

By Theorem \ref{THMDW2}, (\ref{EQBS1})  is equivalent to 
\begin{align}
&\frac{\partial}{\partial t} \mathcal{W}_{L,+}(\bar{g}_{\infty}(t), f_{\infty}(t),t)=0
\\ & \delta(\Rc_{G_{\infty}}+ \frac{\mathbb{1}}{2t})=0,
\end{align}
where $\mathcal{W}_{L,+}$ is Lott's functional given in Definition \ref{DEFWL}.
 Hence the conclusions of Proposition \ref{PROPWLCON} and Corollary \ref{CORWLCON} and the condition $\delta(\Rc_{G_{\infty}}+ \frac{\mathbb{1}}{2t})=0$ hold true in the blowdown limit. Given that $\mathcal{G}$ is $Nil^{3}$, these conditions are equivalent to Parts 1)-8) and a)-c) in the theorem.  
\end{proof}
We then have
\begin{thm}\label{THMSOL1}
The blowdown limit in the setup of Theorem \ref{THMBLOW2} is locally an expanding Ricci soliton. 
\end{thm}

\begin{rmk}
The blowdown limit in the abelian setting is known to locally be an expanding Ricci soliton \cite{LOTT2}. One may also use Lemma \ref{LEMSOL1} below to prove this result. 
\end{rmk}

The proof  of Theorem \ref{THMSOL1} follows directly from Theorem \ref{THMBLOW2}, Remark \ref{RMKBS1} and the following lemma, which holds true for a general nilpotent structure group.

\begin{lemma} \label{LEMSOL1}
Let $\bar{g}(t)$ be an invariant Ricci flow defined for $t\in (0, \infty)$ on a nilpotent $\mathcal{G}$-principal bundle $P \rightarrow M$ such that
\begin{align}
& F_{\bar{A}(t)}=0  \hspace{ 18mm} \frac{\partial A}{\partial t}=0
\\& \frac{\partial G}{\partial t}=-2\Ric_{G}  \hspace{ 10mm} g(t)=tg|_{t=1}
\\& \delta(G^{-1}DG)=0 \hspace{ 8mm} \delta(\Rc_{G} + \frac{\mathbb{1}}{2t})=0.
\end{align}
Choose a trivialization $\phi: U \times \mathcal{G} \rightarrow P$ such that $\phi^{*}\bar{A}$ is the standard flat connection on $U \times \mathcal{G}$ and let $\hat{g}(t)$ be the natural pullback of $\phi^{*}\bar{g}(t)$ to $U \times \mathcal{\tilde{G}}$, where $\tilde{\mathcal{G}}$ is the universal cover of $\mathcal{G}$. Then $\hat{g}(t)$ is an expanding Ricci soliton. 
\end{lemma}
\begin{proof}

After taking the trivialization in the lemma, we may assume: 

$\bullet$ $\mathcal{G}=(\mathbb{R}^{n}, \cdot)$ is a nilpotent Lie group with Lie algebra $\mathfrak{g}=(\mathbb{R}^{n}, [,])$

$\bullet$ $\bar{g}(t)$ is an invariant Ricci flow for $t\in (0, \infty)$ on the principal bundle $\pi: P=M \times \mathcal{G} \rightarrow M$ such that  
\begin{align} \label{EQGSPLIT}
\bar{g}(t)|_{(m,x)}=   g(t)|_{m} \oplus   \tilde{G}_{m}(t)|_{x} 
\end{align} 
based on $T_{(m,x)}P= T_{m}M \oplus T_{x}\mathcal{G}$, where $\tilde{G}_{m}(t):= \bar{g}(t)|_{\pi^{-1}(m)}$. Moreover, we may suppose 
\begin{align}
&\label{EQS1} \frac{\partial G}{\partial t}=-2\Ric_{G}  
 \\ \label{EQS2} &g(t)=tg|_{t=1}
\\& \label{EQS3} \delta(G^{-1}DG)=0 
\\ \label{EQS4}& D \Rc_{G}=0
\\  \label{EQS5}& \delta(\Rc_{G} + \frac{\mathbb{1}}{2t})=0.
\end{align}
Note that (\ref{EQS4}) follows from (\ref{EQS3}) and (\ref{EQDRC2}).

Given this data, we will first show that there exists $\psi_{t} \in \Aut \mathcal{G}$ such that 
\begin{align} \label{EQSOL1}
 \tilde{G}_{m}(t)=t \psi_{t}^{*}\big{(}\tilde{G}_{m}|_{t=1} \big{)},
\end{align}
for all $m\in M$. 

For this, consider the isomorphism
\begin{align*}
\rho: \mathfrak{G} \rightarrow M \times \mathfrak{g}
\end{align*}
given by $\rho(\{s(m),x\})=(m,x)$, where $x \in \mathfrak{g}$ and $s:M\rightarrow P=M\times \mathcal{G}$ is defined by $s(m)=(m,0).$ Under this isomorphism, the connection $D$ on $\mathfrak{G}$ corresponds to the trivial connection on $M \times \mathfrak{g}$. This follows from (\ref{EQGSPLIT}) and Part c) of Lemma \ref{PROPBRA}. Moreover, we have that 
\begin{align*}
 \rho|_{m}(\eta)= (m, \tilde{\eta}|_{(m,0)}),
\end{align*}
for all $\eta \in \mathfrak{G}|_{m}$ and $m \in M$. Here we are using the ``tilde"-notation of Section \ref{SECRFP1}, defined by  
\begin{align*}
\tilde{\eta}|_{p}:= \frac{d}{dt}|_{t=0} \hspace{.3mm} p\cdot exp(tx),
\end{align*}
 for $\eta=\{p,x\} \in \mathfrak{G}$, $p\in P$ and $x \in \mathfrak{g}$.

This together with (\ref{EQS4}) and (\ref{EQS5}) imply that  
\begin{align*} 
B|_{m}:=-\Rc_{\tilde{G}_{m}}|_{t=1} -\frac{\mathbb{1}}{2}, 
\end{align*}
 which is implicitly evaluated at $0\in \mathcal{G}$, is a derivation of $\mathfrak{g}$ that is independent of $m \in M$. We will let $B:=B|_{m}$ for any $m \in M$. 

 If we define $\psi_{t}=exp(ln t \hspace{.3mm} B)$, which is an element of $\Aut{\mathfrak{g}}=\Aut \mathcal{G}$, then by Lemma \ref{LEMSOL2},  $t \psi_{t}^{*}\big{(}\tilde{G}_{m}|_{t=1} \big{)}$ is an invariant Ricci flow on $\mathcal{G}$. Since, by (\ref{EQS1}), $\tilde{G}_{m}(t)$ is also a Ricci flow, (\ref{EQSOL1}) holds true. 

Defining
 \begin{align*} & \hat{\psi}_{t}: M\times \mathcal{G} \rightarrow M \times \mathcal{G}
  \\& \hat{\psi}_{t}(m,x)=(m, \psi_{t}(x)),
\end{align*}
it then follows from  (\ref{EQGSPLIT}), (\ref{EQS2}) and (\ref{EQSOL1}) that 
\begin{align}
\bar{g}(t)= t \hat{\psi}_{t}^{*} (\bar{g}|_{t=1}).
\end{align}
Hence $\bar{g}(t)$ is an expanding Ricci soliton  on $P$.
\end{proof}
\section{Blowdown Limits and Solitons: Four Dimensional Case}\label{SECBLOWSOL2}

In this section, we study Ricci flow blowdown limits on $Nil^{3}$-principal bundles with one dimensional base manifolds. We show that the limit is locally an expanding Ricci soliton and  classify its solution. Combining this with Lott's work in the abelian setting \cite{LOTT2}, yields a local classification of invariant Ricci flow blowdown limits on four dimensional, nilpotent principal bundles fibered over compact base manifolds. 

We have the following theorem:
\begin{thm}\label{THMFOURSOL}
Let $\hat{g}(t)$ be an invariant Ricci flow solution   defined for  $t \in (0, \infty)$       on a $Nil^{3}$-principal bundle $P\rightarrow M$, over a compact one dimensional manifold. Define $\hat{g}_{i}(t)=s_{i}^{-1}\hat{g}(s_{i}t)$, where $\{s_{i}\}$ is a sequence of positive numbers such that $\lim s_{i}=\infty$. We suppose there exists a $Nil^{3}$--principal bundle $P_{\infty} \rightarrow M_{\infty}$, over a compact one dimensional  manifold, together with bundle isomorphisms $\psi_{i}:P_{\infty} \rightarrow P$  such that $\psi_{i}^{*}\hat{g}_{i}(t)$  converges smoothly and uniformly to $\bar{g}_{\infty}(t)$ on $P_{\infty} \times [j^{-1},j]$, for all $j>0$. The following holds true:

1) $\bar{g}_{\infty}(t)$ is locally an expanding Ricci soliton.

2) There exists a trivialization $\phi: \tilde{U} \times Nil^{3} \rightarrow P_{\infty}$, a coordinate $s: \tilde{U}\rightarrow U:=s(\tilde{U})$ and a basis $\{x_{i}\}$ of $\mathfrak{nil}^{3}$, where $U$ is an open interval about $0$ in $\mathbb{R}$ and $x_{3} \in Z(\mathfrak{nil}^{3})$,  such that $\bar{g}(t)= ( \phi \circ (s^{-1}, \mathbb{1}) )^{*}\bar{g}_{\infty}(t)$ on $U \times Nil^{3}$ satisfies: 
\begin{align}
\label{EQSOL4A}\bar{g}(t)|_{(s,0)}(x_{i},x_{j})&=
\begin{pmatrix}
 t^{\frac{1}{3}}e^{sX} & 0_{2 \times 1} \\
0_{1 \times 2} & t^{-\frac{1}{3}} 
\end{pmatrix}_{ij}
\\ \label{EQSOL4B} g(t)&=\frac{t}{2}\tr X^{2}\hspace{.2mm} ds^{2},
\end{align}
where $(s,0) \in U \times Nil^{3}$, given that $Nil^{3}= (\mathbb{R}^{3}, \cdot)$, and $X$ is a $2\times 2$ diagonal matrix with $\tr X=0.$

\end{thm}

\begin{rmk}
See \cite{LOTT2} for the case when the structure group is abelian.
\end{rmk}

\begin{proof}
Since $dim M=1$, the $F=0$ assumption in Theorem \ref{THMBLOW2} holds true. Part 1) of Theorem \ref{THMFOURSOL} then follows from Theorem \ref{THMSOL1}.

As for Part 2), using that $\frac{\del \bar{A}_{\infty}}{\del t}=0$ and $F_{\bar{A}_{\infty}}=0$, choose a trivialization $\phi:\tilde{U} \times Nil^{3} \rightarrow P_{\infty}$ such that $\phi^{*}\bar{A}_{\infty}$ is the standard flat connection on $\tilde{U} \times Nil^{3}$. Also let $s: \tilde{U}\rightarrow U:=s(\tilde{U})$ be a coordinate so that $U$ is an open interval about $0$ in $\mathbb{R}$ and $g_{ss}$ is constant on $U$, 
where $(s^{-1})^{*}g_{\infty}(t)=g_{ss}(t)ds^{2}$.

We will first derive (\ref{EQSOL4A}) and (\ref{EQSOL4B}) for $\bar{g}(t)= \big( \phi \circ (s^{-1}, \mathbb{1}) \big) ^{*}\bar{g}_{\infty}(t)$ on $P= U \times Nil^{3} \rightarrow U$ when $t=1$. Let $\psi:U \rightarrow P$ be the section defined by $\psi(s)=(s,0)$, where $0 \in Nil^{3}= (\mathbb{R}^{3}, \cdot)$. 
Choose a basis $\{x_{i} \}$ for $\mathfrak{nil}^{3}$ such that $\lambda_{i}:= \{\psi,x_{i} \} \in \Gamma(\mathfrak{G})$ and $G_{ij}:=G(\lambda_{i},\lambda_{j})$ satisfy: 
\begin{align}
& Z(\mathfrak{G})= Span(\lambda_{3})
\\& Z(\mathfrak{G})^{\perp}= Span(\lambda_{1}, \lambda_{2})
\\& G_{ij}|_{s=0}=\delta_{ij},
\end{align}
where the terms above and below are to be evaluated at $t=1$. Here we used Part 2) of Theorem \ref{THMBLOW2} and, if it was necessary, replaced $U$ with a smaller interval about $0$ in $\mathbb{R}$. 

The following then holds true:
\begin{align}
 \label{EQSOL41} G_{33}&=1 \ \ \text{on} \ U
\\ \label{EQSOL42} \det G_{uv}&=1 \ \ \text{on} \ U
\\ \label{EQSOL43} \frac{d^{2}G_{uv}}{ds^{2}}&= \frac{dG_{uq}}{ds}\frac{dG_{vr}}{ds} G^{qr}
\\ \label{EQSOL44} 2g_{ss}&=\frac{dG_{uv}}{ds}\frac{dG_{qr}}{ds}G^{uq}G^{vr},
\end{align}
where the  $u,v,q,r$-indices range over $\{1,2\}$. These equations follow from Parts 3), 4)  and 5) of Theorem \ref{THMBLOW2}. In addition, one uses Lemma \ref{propGRAD} when deriving (\ref{EQSOL42}) and that  $g_{ss}$ is constant on $U$ for (\ref{EQSOL43}). 

The general solution of (\ref{EQSOL41})-(\ref{EQSOL44}) for a positive definite matrix $G_{ij}(s)$ with $G_{ij}|_{s=0}=\delta_{ij}$  is:
\begin{align}
& G_{ij}(s)= \begin{pmatrix}
G_{uv}(s) & 0_{2 \times 1} \\
0_{1 \times 2} & 1 
\end{pmatrix}
\\& G_{uv}(s)= (e^{sX})_{uv} \text{ and } g_{ss}=\frac{1}{2}\tr X^{2},
 \end{align}
where $X_{uv}$ is a symmetric $2\times 2$ matrix with $\tr X=0$. 

Next, choose an orthogonal $2 \times 2$ matrix $A$ such that $\hat{X}=AXA^{T}$ is diagonal and define
\begin{align}
 \hat{x}_{u}&=A_{uv}x_{v} \ \  \ \  \hat{x}_{3}=x_{3}
 \\ \hat{\lambda}_{u}&= \{\psi, \hat{x}_{u} \} \ \   \hat{\lambda_{3}}= \lambda_{3},
\end{align}
 where $u,v \in \{1,2\}$. 

In the $\{\hat{\lambda}_{i}\}$-frame, 

\begin{align}
G(\hat{\lambda}_{i},\hat{\lambda}_{j})&=
\begin{pmatrix}
 e^{s\hat{X}} & 0_{2 \times 1} \\
0_{1 \times 2} & 1 
\end{pmatrix}_{ij}
\\ g_{ss}&=\frac{1}{2}\tr \hat{X}^{2},
\end{align}
which holds true at $t=1$. 
Using Parts 8), b) and c) of Theorem \ref{THMBLOW2}, completes the proof of the theorem. 
\end{proof}

\bibliographystyle{plain}

\end{large}
\end{document}